\documentclass[leqno,11pt]{amsart}
\usepackage{amsthm}
\usepackage{amsmath}
\usepackage{amssymb}
\usepackage{amsbsy}
\usepackage{amsfonts}

\usepackage[all]{xy}

\usepackage{hyperref}

\theoremstyle{plain}

\newtheorem{theorem}{Theorem}[section]
\newtheorem{corollary}[theorem]{Corollary}
\newtheorem{lemma}[theorem]{Lemma}
\newtheorem{remark}[theorem]{Remark}

\newtheorem{proposition}[theorem]{Proposition}

\theoremstyle{remark}

\newtheorem*{example-it}{Example}

\newcommand{\ab}{\operatorname{ab}}
\newcommand{\Conj}{\operatorname{Conj}}
\newcommand{\id}{\operatorname{id}}
\newcommand{\im}{\operatorname{im}}
\newcommand{\Jac}{\operatorname{Jac}}
\newcommand{\pr}{\operatorname{pr}}
\newcommand{\tr}{\operatorname{tr}}
\newcommand{\Tr}{\operatorname{Tr}}
\newcommand{\ver}{\operatorname{ver}}

\newcommand{\cO}{\mathcal{O}}
\newcommand{\G}{\mathcal{G}}
\newcommand{\R}{\mathcal{R}}
\newcommand{\Z}{\mathcal{Z}}

\newcommand{\Q}{{\mathbb{Q}}}

\begin{document}

\title[{{$K_1$ of certain Iwasawa algebras, after Kakde}}]{\textbf{$K_1$ of certain Iwasawa algebras, \\ after Kakde}}
\author{Peter Schneider and Otmar Venjakob}
\address{Universit\"{a}t M\"{u}nster,  Mathematisches Institut,  Einsteinstr. 62,
48291 M\"{u}nster,  Germany,
 http://www.uni-muenster.de/math/u/schneider/ }%
\email{pschnei@uni-muenster.de }%

\address{Universit\"{a}t Heidelberg,  Mathematisches Institut,  Im Neuenheimer Feld 288,  69120
Heidelberg,  Germany,
 http://www.mathi.uni-heidelberg.de/$\,\tilde{}\,$venjakob/}
\email{venjakob@mathi.uni-heidelberg.de}
\date{June 15, 2011}
\maketitle

This paper contains a detailed exposition of the content of section five in Kakde's paper \cite{Kak}. We proceed in a slightly more axiomatic way to pin down the exact requirements on the $p$-adic Lie group under consideration. We also make use of our conceptual theory of the completed localization of an Iwasawa algebra as developed in \cite{SV1}. This simplifies some of the arguments. Otherwise, with the exception of the notation at certain places, we follow Kakde's paper.

Let $\G$ be a one dimensional compact $p$-adic Lie group and $\Lambda(\G)$ its Iwasawa algebra. The main purpose is to establish a description of the $K$-group $K_1(\Lambda(\G))$ in terms of the groups of units $\Lambda(U)^\times$ for a suitable system of abelian subquotients $U$ of $\G$. The strategy for achieving this relies on the following commutative diagram
\begin{equation*}
    \xymatrix{
       1  \ar[r] & \mu(\mathcal{O}) \times \G^{\ab} \ar[d]_{=} \ar[r]^-{\iota} & K'_1(\Lambda(\G)) \ar[d]_{\theta} \ar[r]^-{L} & \mathcal{O}[[\Conj(\G)]] \ar[d]_{\beta}^{\cong} \ar[r]^-{\omega} & \G^{\ab} \ar[d]_{=} \ar[r] & 1  \\
     1 \ar[r] & \mu(\mathcal{O}) \times \G^{\ab} \ar[r]^-{\theta\circ \iota} & \Phi \ar[r]^-{\mathcal{L}} & \Psi \ar[r]^-{\omega\circ\beta^{-1}} & \G^{\ab} \ar[r] & 1 }
\end{equation*}
which to explain in detail is not necessary at this point. We only mention that
$\mathcal{O}[[\Conj(\G)]]$ is an additive version of $K_1(\Lambda(\G))$, $L$ is the integral
logarithm of Oliver and Taylor, and $\Phi$ and $\Psi$ are the description we want to achieve and
its additive version, respectively. The upper row is exact by work of Oliver (and Fukaya/Kato). In
a first step (section \ref{sec:additive}) the somewhat easier additive isomorphism $\beta$ will be
established. Then it will be a major point to define the map $\mathcal{L}$ which makes the middle
square commutative. Finally it will be shown that the lower row is exact as well. For both see
section \ref{sec:multiplicative}. Once all of this is done it is a formal consequence that the map
$\theta$ is an isomorphism. Some of this will be generalized to the completed localization $B(\G)$
of $\Lambda(\G)$, additively in section \ref{sec:additive} and multiplicatively in section
\ref{sec:multiplicative-B}. The latter requires an extension of the integral logarithm to the
$B(\G)$-setting (section \ref{sec:log-B}).

\section{Skew Laurent series}

All Iwasawa algebras under consideration will have coefficients in the ring of integers $\mathcal{O}$ of a finite extension of ${\Q}_p$. We fix a prime element $\pi \in \mathcal{O}$. Let $\mathcal{G}$ be any compact $p$-adic Lie group. We \textit{assume} that $\mathcal{G}$ contains a closed normal subgroup $H$ such that $\mathcal{G} / H \cong {\mathbb{Z}}_p$.

In order to give a description of $\Lambda (\mathcal{G})$ relative to $\Lambda (H)$ we choose a closed subgroup $\Gamma \subseteq \mathcal{G}$ such that $\Gamma \xrightarrow{\; \cong \;} \mathcal{G} / H$ and we pick a topological generator $\gamma$ of $\Gamma$ (in particular, $\mathcal{G} = H \rtimes \Gamma$). We then have the ring automorphism
\begin{align*}
\sigma: \Lambda (H) & \longrightarrow \Lambda (H)\\
a & \longmapsto \gamma a \gamma^{-1}
\end{align*}
as well as the left $\sigma$-derivation $\delta : = \sigma - \id$ on $\Lambda (H)$. We recall that the latter means that
\begin{equation*}
\delta (ab) = \delta (a) b + \sigma (a) \delta (b) \qquad \textrm{for any} \ a, b \in \Lambda (H) \ .
\end{equation*}

\begin{proposition}\label{skew-power}
(Venjakob) We have the isomorphism
\begin{equation*}
\Lambda (H) [[ t; \sigma, \delta ]] = \{ \sum_{i \ge 0 } a_i t^i : a_i \in \Lambda (H) \} \xrightarrow{\; \cong \;} \Lambda (\mathcal{G})
\end{equation*}
sending $t$ to $\gamma - 1$ between the $(\sigma, \delta)$-skew power series ring over $\Lambda (H)$ and the Iwasawa algebra $\Lambda (\mathcal{G})$; the multiplication on the left hand side is determined by the rule
\begin{equation*}
ta = \sigma (a) t + \delta (a) \qquad \textrm{for any} \ a \in \Lambda (H)
\end{equation*}
and by continuity.
\end{proposition}

According to \cite{CFKSV} the set
\begin{equation*}
S : = S (\mathcal{G}) : = \{ f \in \Lambda (\mathcal{G}) : \Lambda (\mathcal{G}) / \Lambda (\mathcal{G}) f \ \textrm{is finitely generated over} \ \Lambda (H) \}
\end{equation*}
is an Ore set in $\Lambda (\mathcal{G})$ consisting of regular elements. We then may form the localization $A(\mathcal{G}) := \Lambda (\mathcal{G})_S$ as well as its $\Jac (\Lambda (H))$-adic completion
\begin{equation*}
B (\mathcal{G}) : = \widehat{\Lambda ({\mathcal{G}})_S} \ .
\end{equation*}

\begin{theorem}\label{skew-Laurent}
(\cite{SV1} Thm.\ 4.7 and Prop.\ 2.26 (i))
\begin{itemize}
\item[i.] The isomorphism in Prop.\ \ref{skew-power} extends to an isomorphism between $\Lambda (H) \ll t; \sigma, \delta]] : = \{ \sum_{i \in {\mathbb{Z}}} a_i t^i : a_i \in \Lambda (H),\, \lim_{i \rightarrow \infty} a_i = 0 \ \textrm{in} \ \Lambda (H) \}$ and $B (\mathcal{G})$.
\item[ii.] $B ( \mathcal{G})$ is noetherian and pseudocompact with
\begin{equation*}
\Jac (B (\mathcal{G})) = \{ \sum_{i \in {\Z}} a_i t^i \in B (\mathcal{G}) : a_i \in \Jac (\Lambda (H)) \} \ .
\end{equation*}
\item[iii.] $B (\mathcal{G})$ is flat over $A(\mathcal{G})$ and hence over $\Lambda (\mathcal{G})$.
\end{itemize}
\end{theorem}

We point out that the commutation rules in the ring $\Lambda (H) \ll t; \sigma, \delta]]$ are considerably more complicated. For example, one has
\begin{equation*}
at^{-1} = \sum_{i < 0} t^i \sigma \delta^{-i -1} (a) \qquad \textrm{for any} \ a \in \Lambda (H)
\end{equation*}
involving an infinite Laurent series.

For later applications we note that, for finite $H$, the pseudocompact topology on $B (\mathcal{G})$ is the $\pi$-adic one.

Let $U \subseteq \mathcal{G}$ be an open subgroup (equipped with $H \cap U$).

\begin{proposition}\label{finitefree}
(\cite{SV1} Prop.\ 4.5) We have
\begin{equation*}
A(\mathcal{G}) = A(U) \otimes_{\Lambda (U)} \Lambda (\mathcal{G}) \quad \textrm{and} \quad B (\mathcal{G}) = B (U) \otimes_{\Lambda (U)} \Lambda (\mathcal{G})
\end{equation*}
as bimodules; in particular, $A(\mathcal{G})$, resp.\ $B (\mathcal{G})$, is free as an $A(U)$-, resp.\ a $B (U)$-, module of rank equal to $[\mathcal{G}: U]$.
\end{proposition}

In the special case where $H$ is finite we find an open subgroup $\Gamma_0 \subseteq \Gamma$ which centralizes $H$. Hence $\Gamma_0$ is open in $\mathcal{G}$ and lies in the center of $\mathcal{G}$. It follows that
\begin{itemize}
\item[--] $A(\mathcal{G})$ is an $A(\Gamma_0)$-algebra which is finitely generated free as a module, and
\item[--] $B (\mathcal{G})$ is a $B (\Gamma_0)$-algebra which also is finitely generated free as a module.
\end{itemize}
It is a special case of Thm.\ \ref{skew-Laurent} that $B (\Gamma_0)$ is a complete discrete valuation ring with prime element $\pi$. Using the Weierstrass preparation theorem one sees directly that $A(\Gamma_0)$ is the localisation of $\Lambda (\Gamma_0)  = \mathcal{O} [[ t ]]$ in the height one prime ideal $\pi \Lambda (\Gamma_0)$ and hence is a discrete valuation ring with completion $B (\Gamma_0)$.

\section{Additive commutators}

For any ring $A$ and any set $X$ we let $A[X]$ denote the free $A$-module over the basis $X$. By $[A,A]$ we denote the additive subgroup of $A$ generated by all additive commutators $[a_1,a_2]$ with $a_1, a_2 \in A$. We note that, if $A_0$ is the center of $A$, then $[A,A]$ is an $A_0$-submodule of $A$. Finally, for any group $G$ we denote by $\Conj(G)$ the set of all conjugacy classes $[g]_G$ of elements $g \in G$.

\begin{lemma}\label{conj}
For any finite group $G$ we have
\begin{equation*}
    \cO[G]/\big[ \cO[G], \cO[G] \big] \xrightarrow{\; \cong \;} \cO[\Conj(G)]
\end{equation*}
as $\cO$-modules.
\end{lemma}
\begin{proof}
We consider the surjective $\cO$-module homomorphism
\begin{align*}
    \cO[G] & \longrightarrow \cO[\Conj(G)] \\
    g & \longmapsto [g]_G \ .
\end{align*}
Because of $gh - hg = gh - h(gh)h^{-1}$ it has $\big[ \cO[G], \cO[G] \big]$ in its kernel. On the other hand let $\sum_{g \in G} a_g g$ be any element in the kernel. It is the sum of the elements $\sum_{g \in [h]_G} a_g g$ in the kernel. In particular, $\sum_{g \in [h]_G} a_g = 0$. Hence
\begin{align*}
    \sum_{g \in [h]_G} a_g g & = \sum_{g \in [h]_G} a_g (g-h) = \sum_{g \in G/G_h} a_{ghg^{-1}} (ghg^{-1} - h) \\
    & = \sum_{g \in G/G_h} a_{ghg^{-1}} ((gh)g^{-1} - g^{-1} (gh)) \in \big[ \cO[G], \cO[G] \big]
\end{align*}
(where $G_h$ denotes the centralizer of $h$ in $G$).
\end{proof}

For our Lie group $\G$ we put
\begin{equation*}
    \cO [[\Conj(\G)]] := \varprojlim \cO[\Conj(\G/N)]
\end{equation*}
where $N$ runs over all open normal subgroups of $\G$. We then obtain in the projective limit an $\cO$-module isomorphism
\begin{equation*}
    \Lambda(\G)/\overline{[\Lambda(\G),\Lambda(\G)]} \xrightarrow{\; \cong \;} \cO[[\Conj(\G)]] \ .
\end{equation*}
If $\G_0 \subseteq \G$ denotes the center then this, in an obvious way, is a $\Lambda(\G_0)$-module isomorphism. Suppose that $\G_0$ is open in $\G$ and put $G := \G/\G_0$. Then $\G_0$ acts by multiplication on $\Conj(\G)$ with finitely many orbits. In fact,
\begin{align*}
    \G_0 \backslash \Conj(\G) & \xrightarrow{\; \simeq \;} \Conj(G) \\
    [g]_\G & \longmapsto [g\G_0]_G
\end{align*}
is a bijection. Hence, noncanonically,
\begin{equation*}
    \cO[[\Conj(\G)]] \cong \Lambda(\G_0)[\Conj(G)]
\end{equation*}
as $\Lambda(\G_0)$-modules.

\begin{lemma}\label{B-comm}
If $\G_0$ is open in $\G$ then we have:
\begin{itemize}
  \item[i.] $[\Lambda(\G), \Lambda(\G)]$ is closed in $\Lambda(\G)$;
  \item[ii.] $B(\G)/[B(\G), B(\G)] = B(\G_0) \otimes_{\Lambda(\G_0)} \cO[[\Conj(\G)]]$.
\end{itemize}
\end{lemma}
\begin{proof}
i. With $\Lambda(\G)$ also $[\Lambda(\G), \Lambda(\G)]$ is a finitely generated $\Lambda(\G_0)$-module. Hence $[\Lambda(\G), \Lambda(\G)]$ is the image of a continuous $\Lambda(\G_0)$-linear map $\Lambda(\G)^m \longrightarrow \Lambda(\G)$ between compact modules.

ii. By Prop.\ \ref{finitefree} we have $B(\G) = B(\G_0) \otimes_{\Lambda(\G_0)} \Lambda(\G)$. It follows that $[B(\G), B(\G)] = B(\G_0) \cdot [\Lambda(\G), \Lambda(\G)]$.
\end{proof}

\section{The additive theory}\label{sec:additive}

Let $U \subseteq \G$ be any open subgroup. On the one hand we let $N(U)$ denote the normalizer of $U$ in $\G$, and we put $W(U) := N(U)/U$. On the other hand we note that, by \cite{DDMS} Thm.\ 8.3.2 and Prop.\ 1.19, the commutator subgroup $[U,U]$ is closed in $U$. Hence
\begin{equation*}
    U^{\ab} := U/[U,U]
\end{equation*}
is a commutative Lie group. We let $\pr^U_{U^{\ab}}:\cO[[\Conj(U)]] \twoheadrightarrow \cO[[\Conj(U^{\ab})]]=\Lambda(U^{\ab})$ denote the canonical surjection induced by the obvious projection $U \rightarrow U^{\ab}$.

Let $U \subseteq V \subseteq \G$ be open subgroups. Any $g \in V$ acts by left multiplication on $V/U$. We note:
\begin{itemize}
  \item[--] Any set of representatives for the left cosets of $U$ in $V$ also is a basis of $\Lambda(V)$ as a right $\Lambda(U)$-module.
  \item[--] We have $gxU = xU$ if and only if $x^{-1}gx \in U$.
\end{itemize}
Since $gx = x(x^{-1}gx)$ the trace map in this situation therefore is given by
\begin{align*}
    \tr^V_U : \cO[[\Conj(V)]] & \longrightarrow \cO[[\Conj(U)]] \\
    [g]_V & \longmapsto \sum_{x \in V/U , x^{-1}gx \in U} [x^{-1}gx]_U \ .
\end{align*}
It is $\Lambda(\G_0 \cap U)$-linear. If $[V,V]\subseteq U$, we shall also need the two maps
\begin{equation*}
    \xymatrix{
       & & & \Lambda(V^{\ab}) \ar[d]^{\tau^V_U:=\tr^{V^{\ab}}_{U/[V,V]}} \\
      \Lambda(U^{\ab}) \ar[rrr]^{\pi^V_U:=\pr^{U^{\ab}}_{U/[V,V]}} & & & \Lambda(U/[V,V]) ,  }
\end{equation*}
where $\pr^{U^{\ab}}_{U/[V,V]}$ is the obvious canonical surjection.

\begin{remark}\label{trace}
\begin{itemize}
  \item[i.] If $U$ is normal in $V$ we have
\begin{equation*}
     \tr^V_U ([g]_V) =
     \begin{cases}
     \sum_{x \in V/U} [xgx^{-1}]_U & \textrm{if $g \in U$}, \\
     0 & \textrm{if $g \not\in U$}.
     \end{cases}
\end{equation*}
  \item[ii.] If $V$ is abelian we have
\begin{equation*}
     \tr^V_U (g) =
     \begin{cases}
      [V:U] g & \textrm{if $g \in U$}, \\
     0 & \textrm{if $g \not\in U$}.
     \end{cases}
\end{equation*}
 \item[iii.] If $[V,V]\subseteq U$ then \begin{equation*}
     \tau^V_U (g[V,V]) =
     \begin{cases}
      [V:U] g[V,V] & \textrm{if $g \in U$}, \\
      0 & \textrm{if $g \not\in U$}.
     \end{cases}
\end{equation*}
\end{itemize}
\end{remark}

Suppose that $U$ is abelian. Then:
\begin{itemize}
  \item[--] For any $g \in W(U)$ we have the well defined ring automorphism
      \begin{align*}
        \sigma_{U,g} : \Lambda(U) & \longrightarrow \Lambda(U) \\
        f & \longmapsto gfg^{-1}
      \end{align*}
      as well as the $\Lambda(\G_0 \cap U)$-linear endomorphism
      \begin{equation*}
        \sigma_U := \sum_{g \in W(U)} \sigma_{U,g} \ .
      \end{equation*}
  \item[--] The image of $\sigma_U$ is an ideal in the subring $\Lambda(U)^{W(U)} := \{ f \in \Lambda(U) : \sigma_{U,g}(f) = f \ \textrm{for any $g \in W(U)$} \}$.
\end{itemize}

In a later section we will need, more generally, for two open subgroups $U \subseteq V \subseteq \G$ such that $U$ is normal in $V$ the $\Lambda(\G_0 \cap U)$-linear map
\begin{align*}
    \sigma_U^V:\Lambda(U^{\ab}) & \longrightarrow \Lambda(U^{\ab}) \\
     f & \longmapsto \sum_{g\in V/U} gfg^{-1} \ .
\end{align*}

From now on we assume that
\begin{equation}\tag{H1}\label{f:H1}
    \textrm{$\G_0$ is open in $\G$},
\end{equation}
and we fix an open central subgroup $\Z \subseteq \G$. Let $S(\G,\Z)$  denote the set
of all subgroups $\Z \subseteq U \subseteq \G$.
For $U\in S(\G,\Z)$ we define
\begin{equation*}
    \beta_U := \pr^U_{U^{\ab}}\circ \tr^{\G}_U \ .
\end{equation*}
Our main interest in this section lies in the $\Lambda(\Z)$-linear map
\begin{align*}
    \beta := \beta^{\G}_{\Z} : \cO[[\Conj(\G)]] & \longrightarrow \prod_{U \in S(\G,\Z)} \Lambda(U^{\ab}) \\
    f & \longmapsto (\beta_U(f))_U \ .
\end{align*}

We exhibit three conditions any element $(a_U)_U$ in the image of $\beta$ has to satisfy.

\noindent 1. Let  $U \subseteq V$ be in $S(\G,\Z)$ such that $[V,V]\subseteq U$. In particular, $U$ is normal in $V$. We claim that the diagram
\begin{equation*}
    \xymatrix{
      \cO[[\Conj(\G)]] \ar[d]_{\beta_U} \ar[r]^{\beta_V} & \Lambda(V^{\ab}) \ar[d]^{\tau^V_U} \\
      \Lambda(U^{\ab}) \ar[r]^{\pi^V_U} & \cO[[ U/[V,V]]]   }
\end{equation*}
commutes. By the transitivity of traces this reduces to the commutativity of the diagram
\begin{equation*}
    \xymatrix{
      \cO[[\Conj(V)]] \ar[d]_{\tr^V_U} \ar[rrr]^{\pr^V_{V^{\ab}}} &&& \Lambda(V^{\ab}) \ar[d]^{\tr^{V^{\ab}}_{U/[V,V]}} \\
      \cO[[\Conj(U)]] \ar[rrr]^{\pr^{U^{\ab}}_{U/[V,V]} \circ \pr^U_{U^{\ab}}} &&& \Lambda(U/[V,V]) .   }
\end{equation*}
It suffices to check this on classes $[g]_V$ for $g\in V$. Indeed, using Remark \ref{trace}.i/iii we
compute
\begin{align*}
\tr^{V^{\ab}}_{U/[V,V]} \circ \pr^V_{V^{\ab}} ([g]_V) & = [V:U]g[V,V] \\
& = \sum_{x\in V/U} xgx^{-1}[V,V]\\
& = \pr^{U^{\ab}}_{U/[V,V]} \circ \pr^U_{U^{\ab}} (\sum_{x\in V/U} [xgx^{-1}]_U) \\
& = \pr^{U^{\ab}}_{U/[V,V]} \circ \pr^U_{U^{\ab}} \circ \tr^V_U ([g]_V)
\end{align*}
for $g \in U$; if $g \not\in U$ then both sides are equal to zero. Hence:
\begin{equation}\tag{A1}\label{f:A1}
    \tau^V_U (a_V) = \pi^V_U(a_U) \qquad\textrm{for any $U \subseteq V$ in $S(\G,\Z)$ with $[V,V]\subseteq U$.}
\end{equation}
2) For any open subgroup $U \subseteq \G$ and any $g \in \G$ the diagram
\begin{equation*}
    \xymatrix{
                & \cO[[\Conj(\G)]] \ar[dl]_{\tr^\G_U}  \ar[dr]^{\tr^\G_{gUg^{-1}}}             \\
 \cO[[\Conj(U)]] \ar[rr]^{g . g^{-1}} & &   \cO[[\Conj(gUg^{-1})]]        }
\end{equation*}
is commutative. This, in particular, implies
\begin{equation}\tag{A2}\label{f:A2}
    a_{gUg^{-1}} = ga_U g^{-1} \qquad\textrm{for any $U \in S(\G,\Z)$ and $g \in \G$}.
\end{equation}
3) Let $U \in S(\G,\Z)$. By the transitivity of traces we have
\begin{equation*}
    \beta_U = \pr^U_{U^{\ab}} \circ \tr^\G_U = \pr^U_{U^{\ab}} \circ \tr^{N(U)}_U \circ \tr^\G_{N(U)}
\end{equation*}
and hence
\begin{equation*}
    \im (\beta_U) \subseteq \im (\pr^U_{U^{\ab}} \circ \tr^{N(U)}_U) \ .
\end{equation*}
But Remark \ref{trace}.i implies
\begin{equation*}
    \pr^U_{U^{\ab}} \circ \tr^{N(U)}_U ([g]_{N(U)}) =
    \begin{cases}
    \sigma^{N(U)}_U(g[U,U]) & \textrm{if $g \in U$}, \\
    0 & \textrm{if $g \not\in U$}.
    \end{cases}
\end{equation*}
It follows that $\im(\beta_U) \subseteq \im (\sigma^{N(U)}_U)$. We conclude:
\begin{equation}\tag{A3}\label{f:A3}
    a_U \in \im(\sigma^{N(U)}_U) \qquad\textrm{for any $U \in S(\G,\Z)$}.
\end{equation}

Let
\begin{equation*}
    \Psi := \Psi^\G_\Z \subseteq \prod_{U \in S(\G,\Z)} \Lambda(U^{\ab})
\end{equation*}
denote the subgroup of all elements $(a_U)_U$ which satisfy \eqref{f:A1}, \eqref{f:A2}, and \eqref{f:A3}. So far we have shown the following.

\begin{lemma}\label{image-beta}
$\im (\beta) \subseteq \Psi$.
\end{lemma}

Our goal is to show that $\beta$ induces an isomorphism $\cO[[\Conj(\G)]] \cong \Psi$. For this purpose an important technical role is played by the subset $C(\G,\Z)$ of all $U \in S(\G,\Z)$ such that
$U/\Z$ is cyclic. Let $U \in C(\G,\Z)$. Then $U$ can be generated by $\Z$ and at most one
additional element, and hence is abelian. We introduce the $\Lambda(\Z)$-linear maps
\begin{equation*}
    \eta_U : \Lambda(U) = \oplus_{g \in U/\Z}  \Lambda(\Z)g \longrightarrow \Lambda(U)
\end{equation*}
with
\begin{equation*}
    \eta_U | \Lambda(\Z)g :=
    \begin{cases}
    \id & \textrm{if $g$ generates $U/\Z$}, \\
    0 & \textrm{otherwise}
    \end{cases}
\end{equation*}
and
\begin{align*}
    \delta_U : \Lambda(U) & \longrightarrow \cO[[\Conj(\G)]] \otimes \mathbb{Q} \\
    g & \longmapsto \frac{1}{[\G :U]} [g]_\G \ .
\end{align*}
Let
\begin{equation*}
    \pr_{cyc} : \prod_{U \in S(\G,\Z)} \Lambda(U^{\ab}) \longrightarrow \prod_{U \in C(\G,\Z)} \Lambda(U)
\end{equation*}
denote the obvious projection map and define the $\Lambda(\Z)$-linear map
\begin{align*}
    \delta : \prod_{U \in C(\G,\Z)} \Lambda(U) & \longrightarrow \cO[[\Conj(\G)]] \otimes \mathbb{Q} \\
    (a_U)_U & \longmapsto \sum_U \delta_U(\eta_U(a_U)) \ .
\end{align*}

\begin{lemma}\label{inverse}
$\delta \circ \pr_{cyc} \circ \beta = \id$.
\end{lemma}
\begin{proof}
Let $g \in \G$ and put $U_g := \, <g,\Z> \, \in C(\G,\Z)$. We compute
\begin{align*}
    \delta \circ \pr_{cyc} \circ \beta ([g]_\G) & = \sum_{U \in C(\G,\Z)} \delta_U \circ \eta_U \circ \tr^\G_U ([g]_\G) \\
    & = \sum_U \sum _{x \in \G/U, x^{-1}gx \in U} \delta_U \circ \eta_U (x^{-1}gx) \\
    & = \sum_U \sum_{x \in \G/U, x^{-1}U_g x = U} \frac{1}{[\G :U]} [g]_\G \\
    & = \sum_{U\, \textrm{conjugate to}\; U_g} \frac{1}{[\G :N(U_g)]} [g]_\G \\
    & = [g]_\G \ .
\end{align*}

\end{proof}

It follows that the map $\beta$ is injective.

At this point we need further assumptions:
\begin{equation}\label{f:H2}\tag{H2}
    \textrm{$\G$ is a pro-$p$ group.}
\end{equation}
\begin{equation}\label{f:H3}\tag{H3}
    \begin{split}
    \textrm{There is a system of representatives $\R \subseteq \G$ for the cosets} \\
     \textrm{in $\G/\Z$ which contains 1 and consists of full $\G$-orbits. \ }
    \end{split}
\end{equation}
(In the one dimensional case $\G = H \rtimes \gamma^{\mathbb{Z}_p}$ with $\Z = \gamma^{p^e \mathbb{Z}_p}$ we may take $\R := \{ H \times \{\gamma^j : 0 \leq j < p^e \}$.)
For $U \in S(\G,\Z)$ we put $\R_U := \R \cap U$. We also introduce the subgroup
\begin{equation*}
    \Psi_{cyc} \subseteq \prod_{U \in C(\G,\Z)} \Lambda(U)
\end{equation*}
of all elements which satisfy \eqref{f:A1}, \eqref{f:A2}, and \eqref{f:A3} for all (pairs of) subgroups in $C(\G,\Z)$.

\begin{lemma}\label{image-delta}
$\delta | \Psi_{cyc}$ is injective with $\delta (\Psi_{cyc}) \subseteq \cO[[\Conj(\G)]]$.
\end{lemma}
\begin{proof}
Let $(a_U)_U \in \Psi_{cyc}$. We first assume that
\begin{equation*}
    \delta((a_U)_U) = \sum_{U \in C(\G,\Z)} \delta_U (\eta_U(a_U)) = 0 \ .
\end{equation*}
\textit{Step 1:} We show that $\delta_U(\eta_U(a_U)) = 0$ for any $U \in C(\G,\Z)$. By definition $\delta_U(\eta_U(a_U))$ is supported on $\bigcup_{g \in \G} gUg^{-1}$. So, if $U_1, \ldots, U_m$ are representatives for the $\G$-orbits in $C(\G,\Z)$, then we obtain
\begin{equation*}
    \sum_{g \in \G/N(U_i)} \delta_{gU_i g^{-1}} (\eta_{gU_i g^{-1}}(a_{gU_ig^{-1}})) = 0 \qquad\textrm{for any $1 \leq i \leq m$}.
\end{equation*}
But \eqref{f:A2} implies that
\begin{equation*}
    \delta_{gU_i g^{-1}} (\eta_{gU_i g^{-1}}(a_{gU_ig^{-1}})) = g \delta_{U_i} (\eta_{U_i}(a_{U_i})) g^{-1} = \delta_{U_i}(\eta_{U_i}(a_{U_i}))
\end{equation*}
for any $1 \leq i \leq m$ and $g \in \G$.\\
\textit{Step 2:} We show that $a_U = 0$ for any $U \in C(\G,\Z)$. Since
\begin{equation*}
    \Lambda(U) = \oplus_{h \in \R_U} \Lambda(\Z)h
\end{equation*}
we may write
\begin{equation*}
    a_U = \sum_{h \in \R_U} a_h h \qquad\textrm{with $a_h \in \Lambda(\Z)$}.
\end{equation*}
For $g \in N(U)$ we have
\begin{equation*}
    a_U = a_{gUg^{-1}} = ga_U g^{-1} = \sum_{h \in \R_U} a_h ghg^{-1} = \sum_{h \in \R_U} a_{g^{-1}hg} h
\end{equation*}
by \eqref{f:A2} (the last identity is the point where we need the assumption \eqref{f:H3}) and hence $a_{g^{-1}hg} = a_h$. It follows that
\begin{equation*}
    a_U = \sum_{\xi \in N(U) \backslash \R_U} a_\xi (\sum_{h \in \xi} h) \qquad\textrm{with $a_\xi \in \Lambda(\Z)$}.
\end{equation*}
Choosing elements $h_\xi \in \xi$ we then have
\begin{equation*}
    0 = \delta_U(\eta_U(a_U)) = \sum_{\xi \in N(U) \backslash \R_U ,\, U=<h_\xi,\Z>} a_\xi \frac{|\xi|}{[\G : U]} [h_\xi]_\G \ .
\end{equation*}
If $h_\xi$ generates $U/\Z$ then $\xi$ is the intersection with $\R_U$ of a full $\G$-orbit in $\R$. It follows that $a_\xi = 0$ in this case. This shows first of all that $a_\Z = 0$. Secondly, for $U \neq \Z$ we obtain
\begin{equation*}
    a_U = \sum_{h \in \R_{U'}} a_h h
\end{equation*}
where $U' \in C(\G,\Z)$ is such that $U'/\Z$ is the unique subgroup of index $p$ in the cyclic $p$-group $U/\Z$. On the other hand we deduce from \eqref{f:A1} that
\begin{equation*}
    a_{U'} = \tr^U_{U'} (a_U) = p a_U \ .
\end{equation*}
The claim now follows by induction.

This establishes the asserted injectivity. For the claim about the image we note that, by \eqref{f:A3}, we have
\begin{equation*}
    a_U = \sigma_U(b_U) \qquad\textrm{with $b_U \in \Lambda(U)$}
\end{equation*}
for any $U \in C(\G,\Z)$. Hence
\begin{equation*}
    \delta_U(\eta_U(a_U)) = \delta_U(\eta_U(\sum_{g \in W(U)} gb_U g^{-1})) = [N(U):U] \delta_U(\eta_U(b_U)) \ .
\end{equation*}
Using \eqref{f:A2} we further obtain that
\begin{align*}
    \sum_{g \in \G/N(U)} \delta_{gUg^{-1}}(\eta_{gUg^{-1}}(a_{gUg^{-1}})) & = [\G :N(U)] \delta_U(\eta_U(a_U)) \\
    & = [\G :U] \delta_U(\eta_U(b_U))
\end{align*}
lies in $\cO[[\Conj(\G)]]$.
\end{proof}

\begin{lemma}\label{injection}
The map $\pr_{cyc}$ restricts to an injection $\Psi\hookrightarrow\Psi_{cyc}$.
\end{lemma}
\begin{proof}
It is clear that $\pr$ restricts to a map $\Psi \rightarrow \Psi_{cyc}$. To establish its injectivity we argue by contradiction. We assume that $(a_U)_U\in\Psi$ satisfies $a_U=0$ for all $U$ in $C(\G,\Z)$ and that there exists a $V$ in $S(\G,\Z)\setminus C(\G,\Z)$ such that $a_V\neq 0$. Moreover, we may and do assume that $V$ has minimal order $|V/\Z|$ with this property. Let $R \subseteq V$ denote a set of representatives for the cosets in $V/\Z[V,V]$; we write $\bar{h} \in V^{\ab}$ for the image of $h \in R$. We have
\begin{equation*}
    a_V = \sum_{h \in R} a_h \bar{h} \qquad\textrm{with $a_h\in \im(\Lambda(\Z) \rightarrow \Lambda(V^{\ab}))$}.
\end{equation*}
To achieve the contradiction we will show that all coefficients $a_h$ have to vanish. Fix $h_0 \in R$. Since $V/\Z$ is not cyclic (but is a $p$-group) its abelianization $(V/\Z)^{\ab}$ is not cyclic either by \cite{Hup} III.7.1.c. Hence we find a normal subgroup $\Z \subseteq U \unlhd V$ of index $p$ such that $h_0 \in U$. In particular, $[V,V] \subseteq U$. By the minimality of $V$ we have $a_U = 0$. Then
\begin{equation*}
    0 = \tau^V_U (a_V) = \sum_{h \in R} a_h \tau^V_U (\bar{h}) = p \sum_{h \in R \cap U} a_h \bar{h}
\end{equation*}
by \eqref{f:A1} and Remark \ref{trace}.iii. It follows that $a_h = 0$ for any $h \in R \cap U$ and, in particular, $a_{h_0} = 0$.
\end{proof}

The last four lemmas together imply the following fact.

\begin{theorem}\label{iso-additive}
All three maps in the commutative diagram
\begin{equation*}
    \xymatrix@R=0.5cm{
                &     \Psi    \ar[dd]^{\pr_{cyc}}_{\cong}     \\
  \cO[[\Conj(\G)]] \ar[ur]^-{\beta}_-{\cong} \ar[dr]_-{\pr_{cyc}\circ\beta}^-{\cong}                 \\
                &      \Psi_{cyc}                  }
\end{equation*}
are isomorphisms.
\end{theorem}

\begin{remark}\label{qp-version}
Let $\Psi_{\mathbb{Q}_p}$ denote the subset of $\prod_{U\in S(\G,\Z)}
\Lambda(U^{\ab})\otimes_{\mathbb{Z}_p} \mathbb{Q}_p$ consisting of those tuples $(a_U)_U$ which
satisfy the natural analogues of \eqref{f:A1} and \eqref{f:A2}. Note that any such tuple automatically satisfies
$a_U\in \im(\sigma^{N(U)}_U)\otimes_{\mathbb{Z}_p} \mathbb{Q}_p$ for all $U$ in $S(\G,\Z)$, because due to
\eqref{f:A2} one has $\sigma^{N(U)}_U(a_U)=|W(U)| a_U$. Hence $\Psi_{\mathbb{Q}_p}$ can be identified with
\begin{equation*}
    \Psi \otimes_{\mathbb{Z}_p}
\mathbb{Q}_p\subseteq\prod_{U\in S(\G,\Z)} \Lambda(U^{\ab})\otimes_{\mathbb{Z}_p} \mathbb{Q}_p \ .
\end{equation*}
Therefore, by Lemma \ref{injection} the projection $\pr_{cyc} \otimes_{\mathbb{Z}_p} \mathbb{Q}_p$ induces again an injection
\begin{equation*}
    \Psi_{\mathbb{Q}_p}\hookrightarrow \prod_{U\in C(\G,\Z)}
\Lambda(U)\otimes_{\mathbb{Z}_p} \mathbb{Q}_p \ .
\end{equation*}
Now assume that
$(a_U)_U\in \Psi_{\mathbb{Q}_p}$ satisfies   $a_U\in\im(\sigma_U)\subseteq \Lambda(U)$
for all $U$ in $C(\G,\Z)$. Then $(a_U)_U$  belongs already to $\Psi$, in particular all
$a_U$ are integral! Indeed, by the above injectivity   one has
\begin{equation*}
    (a_U)_U=\beta(\pr_{cyc}\circ\beta)^{-1}
(\pr_{cyc}\otimes_{\mathbb{Z}_p} \mathbb{Q}_p) ((a_U)_U)\in\Psi \ .
\end{equation*}
\end{remark}

\begin{remark}\label{eta-A3}
For any pair $W \subseteq V$ of subgroups of index $p$ in $C(\G,\Z)$ we have:
\begin{itemize}
  \item[i.] $\im(\sigma_W) \subseteq p \im(\sigma_V)$;
  \item[ii.] Let $(a_U)_U \in \prod_{U \in C(\G,\Z)} \Lambda(U)$ be an element which satisfies \eqref{f:A1}; then:
      \begin{itemize}
        \item[a.] $a_V = \eta_V(a_V) + \tfrac{1}{p} a_W$;
        \item[b.] $(a_U)_U$ satisfies \eqref{f:A3} if and only if   $(\eta_U(a_U))_U$ satisfies \eqref{f:A3}.
      \end{itemize}
\end{itemize}
\end{remark}
\begin{proof}
i. Using that obviously $N(V) \subseteq N(W)$ we compute
\begin{align*}
    \sigma_W(f) & = \sum_{g \in N(W)/W} gfg^{-1} = \sum_{g \in N(V)/W} \sum_{h \in R} ghfh^{-1}g^{-1} \\
    & = p \sum_{g \in N(V)/V} g(\sum_{h \in R} hfh^{-1})g^{-1} = p \sigma_V(\sum_{h \in R} hfh^{-1})
\end{align*}
for any $f \in \Lambda(W)$, where $R \subseteq N(W)$ is a fixed set of representatives for the right cosets of $N(V)$ in $N(W)$.

ii. By assumption we have $\tr^V_W (a_V) = a_W$. If we write $a_V = \sum_{h \in \R_V} a_h h$ with $a_h \in \Lambda(\Z)$ then, using Remark \ref{trace}.ii, we obtain
\begin{equation*}
    a_W = \tr^V_W (a_V) = \sum_{h \in \R_V} a_h \tr^V_W (h) = p \sum_{h \in \R_W} a_h h = p (a_V - \eta_V(a_V)) \ .
\end{equation*}
Since $\eta_\Z = \id$ the assertion in b. follows by induction from i. and a.
\end{proof}

Using Prop.\ \ref{finitefree} we may extend $B(\Z)$-linearly all the maps which we have considered above. In particular, also using Lemma \ref{B-comm} we have the $B(\Z)$-linear map
\begin{align*}
    \beta_B := \beta^{B(\G)}_{B(\Z)} : B(\G)/[B(\G), B(\G)] & \longrightarrow \prod_{U \in S(\G,\Z)} B(U^{\ab}) \\
    f & \longmapsto ((\id_{B(\Z)} \otimes \beta_U)(f))_U \ .
\end{align*}
In the right hand side we have the subgroup $\Psi_B$ of elements which satisfy the obvious analogs of the conditions \eqref{f:A1} - \eqref{f:A3}. Either by repeating the above arguments or by simply using that, according to Thm.\ \ref{skew-Laurent}.iii, $B(\Z)$ is flat over $\Lambda(\Z)$ we deduce the following consequence.

\begin{theorem}\label{iso-B-additive}
$\beta_B : B(\G)/[B(\G), B(\G)] \xrightarrow{\cong} \Psi_B$ is an isomorphism.
\end{theorem}

\section{The multiplicative theory, part 1}\label{sec:multiplicative}

We suppose $\G$ to satisfy \eqref{f:H1} and \eqref{f:H2}, and we continue to fix an open central subgroup $\Z \subseteq \G$. For any open subgroups $U \subseteq V \subseteq \G$ we have the norm map
\begin{equation*}
    N^V_U : K_1(\Lambda(V)) \longrightarrow K_1(\Lambda(U)) \ .
\end{equation*}
We recall that it is induced by the exact functor which sends a finitely generated projective $\Lambda(V)$-module $P$ to $P$ viewed as a finitely generated projective $\Lambda(U)$-module.
We introduce the composed homomorphisms
\begin{equation*}
    \theta_U : K_1(\Lambda(\G)) \xrightarrow{\; N^\G_U \;} K_1(\Lambda(U)) \longrightarrow K_1(\Lambda(U^{\ab})) = \Lambda(U^{\ab})^\times
\end{equation*}
where the right hand arrow is induced by the canonical surjection $U \twoheadrightarrow U^{\ab}$. If $U$ is abelian then $\theta_U = N^\G_U$, of course. The central object of the multiplicative theory is the homomorphism
\begin{align*}
    \theta := \theta^\G_\Z : K_1(\Lambda(\G)) & \longrightarrow \prod_{U \in S(\G,\Z)} \Lambda(U^{\ab})^\times \\
    x & \longmapsto (\theta_U(x))_U \ .
\end{align*}
Similarly as in the previous section we begin by exhibiting four conditions \eqref{f:M1} -- \eqref{f:M4} which any element $(x_U)_U = \theta(x)$ satisfies.\\
1) Let $\Z \subseteq U \subseteq V \subseteq \G$ be open subgroups such that $[V,V] \subseteq U$. As in the last section we use the ring homomorphism
\begin{equation*}
    \pi^V_U : \Lambda(U^{\ab}) \longrightarrow \Lambda(U/[V,V]) \ .
\end{equation*}
Furthermore, corresponding to the inclusion of groups $U/[V,V] \hookrightarrow V^{\ab}$ we have the norm map
\begin{equation*}
    \nu^V_U := N^{V^{\ab}}_{U/[V,V]} : \Lambda(V^{\ab})^\times \longrightarrow \Lambda(U/[V,V])^\times \ .
\end{equation*}
We claim that the diagram
\begin{equation*}
    \xymatrix{
      K_1(\Lambda(\G)) \ar[d]_{\theta_U} \ar[r]^-{\theta_V} & \Lambda(V^{\ab})^\times \ar[d]^{\nu^V_U} \\
      \Lambda(U^{\ab})^\times \ar[r]^-{\pi^V_U} & \Lambda(U/[V,V])^\times   }
\end{equation*}
is commutative, which implies that
\begin{equation}\tag{M1}\label{f:M1}
    \nu^V_U (x_V) = \pi^V_U (x_U) \qquad\textrm{for any $U \subseteq V$ in $S(\G,\Z)$ with $[V,V] \subseteq U$}
\end{equation}
holds true. We enlarge the above diagram to
\begin{equation*}
    \xymatrix{
      K_1(\Lambda(\G)) \ar[d]_{N^\G_U} \ar[drr]^{N^\G_V}  &&  \\
      K_1(\Lambda(U)) \ar[d]  && K_1(\Lambda(V)) \ar[ll]^-{N^V_U} \ar[dd] \\
      K_1(\Lambda(U^{\ab})) \ar[d] &&  \\
      K_1(\Lambda(U/[V,V])) && \ar[ll]^-{N^{V^{\ab}}_{U/[V,V]}}  K_1(\Lambda(V^{\ab}))   }
\end{equation*}
where the undecorated perpendicular arrows are induced by the obvious ring homomorphisms. The upper triangle is commutative by the transitivity of norm maps. The lower square is commutative because of the identity
\begin{equation*}
    \Lambda(U/[V,V]) \otimes_{\Lambda(U)} \Lambda(V) = \Lambda(V^{\ab}) \ .
\end{equation*}
We point out that if $V$ is abelian then \eqref{f:M1} simplifies to the condition
\begin{equation}\tag{M1a}\label{f:M1a}
    N^V_U(x_V) = x_U \ .
\end{equation}
2) For any open subgroup $U \subseteq \G$ and any $g \in \G$ the diagram
\begin{equation*}
    \xymatrix{
                & K_1(\Lambda(\G)) \ar[dl]_{\theta_U}  \ar[dr]^{\theta_{gUg^{-1}}}             \\
 \Lambda(U^{\ab})^\times  \ar[rr]^{g.g^{-1}} & &    \Lambda((gUg^{-1})^{\ab})^\times        }
\end{equation*}
is commutative. This implies
\begin{equation}\tag{M2}\label{f:M2}
    x_{gUg^{-1}} = gx_U g^{-1} \qquad\textrm{for any $U \in S(\G,\Z)$ and $g \in \G$}.
\end{equation}
3) For the next condition we need the additional assumptions that
\begin{equation}\tag{H4}\label{f:H4}
    \mathcal{O} \ \textrm{ is absolutely unramified and $p \neq 2$}.
\end{equation}
Let $\phi : \mathcal{O} \longrightarrow \mathcal{O}$ denote the Frobenius automorphism.

We first recall the construction of the Verlagerung. Let $U \subseteq V \subseteq \G$ be any two open subgroups, and let $R \subseteq V$ be a set of representatives of the left cosets in $V/U$. Then $R$ also is a basis of $\Lambda(V)$ as a right $\Lambda(U)$-module. For $g \in V$ and $h \in R$ we write
\begin{equation*}
    gh = h_g c_{g,h} \qquad \textrm{with $h_g \in R$ and $c_{g,h} \in U$}.
\end{equation*}
The matrix $M_g$ of left multiplication by $g$ on $\Lambda(V)$ with respect to the basis $R$ is the product of the permutation matrix describing the permutation action of $g$ on the coset space $V/U$ and the diagonal matrix with diagonal entries $\{c_{g,h}\}_{h \in R}$. The map
\begin{align*}\tag{ver}\label{f:ver}
    \ver^V_U : \qquad V^{\ab} & \longrightarrow U^{\ab} \\
    g[V,V] & \longmapsto \prod_{h \in R} c_{g,h} [U,U]
\end{align*}
is a well defined group homomorphism called the transfer map or Verlagerung (cf.\ \cite{Hup} IV.1.4). In the trivial case where $U$ is central in $V$ we have $\ver^V_U(g) = g^{[V:U]}$ for any $g \in V$.

In our case where $V$ is a pro-$p$ group with $p \neq 2$ we have:
\begin{itemize}
  \item[--] $N^V_U(g)$ can be computed as the determinant of $M_g$ in $K_1(\Lambda(U)) = \Lambda(U)^\times/[\Lambda(U)^\times, \Lambda(U)^\times]$.
  \item[--] The determinant of the permutation matrix factor of $M_g$ is equal to one.
\end{itemize}
Hence
\begin{equation}\label{f:N-ver}
    N^V_U(g) \equiv \ver^V_U(g) \bmod [\Lambda(U)^\times, \Lambda(U)^\times] \ .
\end{equation}
In other words we have the commutative diagram
\begin{equation*}
    \xymatrix{
      V^{\ab} \ar[d]_{\ver^V_U} \ar@{^{(}->}[r] & K_1(\Lambda(V)) \ar[d]^{N^V_U} \\
      U^{\ab} \ar@{^{(}->}[r] & K_1(\Lambda(U)).   }
\end{equation*}

In the following we extend $\ver^V_U$ to the unique ring homomorphism
\begin{equation*}
    \ver^V_U : \Lambda(V^{\ab}) \longrightarrow \Lambda(U^{\ab})
\end{equation*}
such that $\ver^V_U | \cO = \phi$.

We also need to introduce, for any open subgroup $U \subseteq
\G$, the unique $\mathbb{Z}_p$-linear continuous map
\begin{equation*}
    \varphi_U : \mathcal{O}[[\Conj(U)]] \longrightarrow \mathcal{O}[[\Conj(U)]]
\end{equation*}
such that
\begin{equation*}
    \varphi_U | \mathcal{O} = \phi \quad \textrm{and} \quad \varphi_U([g]_U) = [g^p]_U \ \textrm{for any $g \in U$}.
\end{equation*}
If $U$ is abelian then $\varphi_U$ is a ring endomorphism of $\Lambda(U)$ with the property that modulo $p$ it coincides with the map $f \longmapsto f^p$.

We now assert:
\begin{equation}\tag{M3}\label{f:M3}
    \ver^V_U(x_V)- x_U \in  \im(\sigma_U^V)\qquad \parbox[t]{6cm}{ for all $U\subseteq V$ in
    $S(\G,\Z)$ such that $[V:U]=p$.}
\end{equation}
Note that in this condition $U$ automatically is normal in $V$. One easily checks that $\ver^U_\Z(\im (\sigma^V_U)) \subseteq p \Lambda(\Z)$. Hence \eqref{f:M3} inductively implies that
$\ver^U_\Z(x_U) = x_\Z \bmod p\Lambda(\Z)$ for any $U \in S(\G,\Z)$. But $\ver^U_\Z$ is just the $|U/\Z|$-power map on the group elements. We therefore deduce:
\begin{equation}\tag{M3a}\label{f:M3a}
    x_U^{|U/\Z|}\equiv x_\Z \bmod p\Lambda(\Z)  \qquad\textrm{for all $U$ in $S(\G,\Z)$}.
\end{equation}

For the proof of \eqref{f:M3} consider as before the diagram
\begin{equation*}
    \xymatrix{
      K_1(\Lambda(\G)) \ar[d]_{N^\G_U} \ar[drr]^{N^\G_V}  &&  \\
      K_1(\Lambda(U)) \ar[d]^{\pr^{U}_{U^{\ab}}}  && K_1(\Lambda(V)) \ar[ll]^-{N^V_U} \ar[d]^{\pr^{V}_{V^{\ab}}} \\
      K_1(\Lambda(U^{\ab}))   && K_1(\Lambda(V^{\ab}))            },
\end{equation*}
from which we see that it suffices to show for  all $x\in\Lambda(V)^\times$ the relation
\begin{equation*}
 \ver^V_U(\pr^{V}_{V^{\ab}}\bar{x} )- \pr^{U}_{U^{\ab}}N^V_U\bar{x} \in  \im(\sigma_U^V)
\end{equation*}
holds, in which $\bar{x}$ denotes the image of $x$ in $K_1(\Lambda(V))$. We choose any $g\in
V\setminus U$ and use the decomposition
\begin{equation}\label{f:decomp}
\Lambda(V)\cong \bigoplus_{i=0}^{p-1} \Lambda(U)g^i
\end{equation}
to write $x=\sum x_k g^k$. In order to calculate $N^V_U\bar{x}$ we denote by $\sigma$ the
endomorphism of $\Lambda(U)$ induced by $u\mapsto gug^{-1}$ and consider the matrix
\begin{equation*}
    \left(
          \begin{array}{llll}
            x_0 & x_1  & \ldots & x_{p-1} \\
            \sigma(x_{p-1})g^p & \sigma(x_0) & \ldots & \sigma(x_{p-2}) \\
            \sigma^2(x_{p-2})g^p & \sigma^2(x_{p-1})g^p & \ldots & \sigma^2(x_{p-3}) \\
            \vdots & \vdots & \ddots & \vdots \\
            \sigma^{p-1}(x_1)g^p & \sigma^{p-1}(x_2)g^p & \ldots & \sigma^{p-1}(x_0) \\
          \end{array}
        \right)
\end{equation*}
describing the $\Lambda(U)$-linear map which arises by multiplication with $x$ from the right using
the basis given in  \eqref{f:decomp} and which thus represents $N^V_U\bar{x}$ in $K_1(\Lambda(U))$.
The image in $\Lambda(U^{\ab})^\times$ therefore is represented by its  determinant. Using the Leibniz rule we obtain
\begin{equation*}
\pr^{U}_{U^{\ab}}N^V_U\bar{x} =
\sum_{\delta\in S}
\mathrm{sign}(\delta)g^{pe_\delta}
\prod_{c \in C} \sigma^c\pr^{U}_{U^{\ab}}x_{\kappa(\delta(c)-c)} \ ;
\end{equation*}
here $S=S(\mathbb{Z}/p\mathbb{Z})$ denotes the symmetric group on (the group) $C:=\mathbb{Z}/p\mathbb{Z}$, $\kappa(c) \in \{0,1,\ldots,p-1\}$ is the respective representative of $c$, and $e_\delta$ is the number of $c \in C$ such that $\kappa(\delta(c))<\kappa(c)$. Note that, since $g^p \in U$, the endomorphism $\sigma^p$ induces the identity on $\Lambda(U^{\ab})$.

The above sum can be decomposed with respect to the following action of $C$ on $S$. Let $\gamma \in S$ denote the cycle of order $p$  defined by $\gamma(c) := c + 1$. We put
\begin{align*}
    C \times S & \longrightarrow S \\
    (c,\delta) & \longmapsto \delta^c := \gamma^c \delta \gamma^{-c} \ .
\end{align*}
Explicitly we have
\begin{equation*}
    \delta^c(c'):=\delta(c'-c)+c \qquad\text{for all $c,c'\in C$}.
\end{equation*}
This action obviously respects the signum of elements in $S$. Moreover, the function $e_\delta$ is constant on each $C$-orbit in $S$. We see that for every $\delta_0 \not\in S^C$ the partial sum
\begin{multline*}
\sum_{\delta\in C\delta_0}
\mathrm{sign}(\delta)g^{pe_\delta}
\prod_{c \in C} \sigma^c\pr^{U}_{U^{\ab}}
x_{\kappa(\delta(c)-c)} \\
= \mathrm{sign}(\delta_0)g^{pe_{\delta_0}}
\sum_{a \in C}\prod_{c \in C} \sigma^c \pr^{U}_{U^{\ab}}x_{\kappa(\delta_0^a(c)-c)} \\
 = \sum_{a \in C} \sigma^a \left(\mathrm{sign}(\delta_0)g^{pe_{\delta_0}}
\prod_{b \in C} \sigma^b \pr^{U}_{U^{\ab}}x_{\kappa(\delta_0(b)-b)}\right)
\end{multline*}
belongs to $\im(\sigma^V_U)$ (the second identity comes from substituting $b$ for $c-a$). On the other hand, if $\delta\in S^C$, one checks immediately that $\delta(c)-c=\delta(0)$ is constant for all $c\in C$, i.e., $\delta = \gamma^{\delta(0)} = \gamma^{\kappa(\delta(0))}$, whence
$\mathrm{sign}(\delta)=1$ (since $p \neq 2$).
Furthermore $e_\delta=\kappa(\delta(0))$. As $S$ decomposes
disjointly  into $S^C$ and the orbits of order $p$, we altogether obtain modulo $\im(\sigma^V_U)$
\begin{align}\label{f:Leibniz}
\pr^{U}_{U^{\ab}}N^V_U\bar{x} & \equiv \sum_{\delta\in S^C}
\mathrm{sign}(\delta)g^{pe_\delta}\prod_{c \in C}^{p-1}
\sigma^c \pr^{U}_{U^{\ab}}x_{\kappa(\delta(c)-c)} \\
& = \sum_{k=0}^{p-1}g^{pk}
\prod_{i=0}^{p-1}\sigma^i\pr^{U}_{U^{\ab}}x_k \ . \nonumber
\end{align}
Using  the fact that $\ver^V_U$ is a ring homomorphism we are now able to determine modulo $\im(\sigma^V_U)$ :
\begin{align*}
\ver^V_U(\pr^{V}_{V^{\ab}}\bar{x} ) & = \sum_{k=0}^{p-1}\ver^V_U(\pr^{V}_{V^{\ab}}(x_k))
\ver^V_U(\pr^{V}_{V^{\ab}}(g)^k) \\
& \equiv \sum_{k=0}^{p-1} \prod_{i=0}^{p-1}\sigma^i(\pr^U_{U^{\ab}}(x_k))
\pr^{U}_{U^{\ab}}(g^p)^k \\
& \equiv \pr^{U}_{U^{\ab}}N^V_U\bar{x} \ .
\end{align*}
Using the set of representatives $R = \{1, g, \ldots, g^{p-1}\}$ one checks from the definition that $\ver^V_U(\pr^{V}_{V^{\ab}}(g)) = \pr^{U}_{U^{\ab}}(g^p)$. The last congruence is \eqref{f:Leibniz} while the middle one can be seen as follows. First recall that $\im(\sigma^V_U)$ is a $\Lambda(U^{\ab})^{V/U}$-ideal and that $\pr^{U}_{U^{\ab}}(g^p)$ belongs to
$\Lambda(U^{\ab})^{V/U}$. Fix any $k$ and write $x_k=\sum_{h\in U/\Z} a_h h$ with $a_h\in\Lambda(\Z)$. Using again the fact that $\ver^V_U$ is a ring homomorphism combined with formula \eqref{f:ver}
(with $c_{u,g^i}=\sigma^{-i}(u)$) we obtain
\begin{align*}
\ver^V_U(\pr^{V}_{V^{\ab}}(x_k)) & = \pr^U_{U^{\ab}}\left(\sum_{h\in U/\Z}
\varphi_\Z(a_h)\prod_{i=0}^{p-1}\sigma^i(h) \right) \\
& \equiv \pr^U_{U^{\ab}}\left(\sum_{h\in U/\Z} a_h^p\prod_{i=0}^{p-1}\sigma^i(h) \right)
\end{align*}
as $\varphi_\Z(a_h)-a_h^p\in p\Lambda(\Z)$, so that its projection to $\Lambda(U^{\ab})$ lies in $\im(\sigma^V_U)$, and $\prod_{i=0}^{p-1}\sigma^i(h)$ belongs to $\Lambda(U)^{V/U}$ for every $ h\in V/U$. Since moreover
\begin{equation*}
    \prod_{i=0}^{p-1}\left(\sum_{h\in U/\Z} a_h\sigma^i(h)\right)-\sum_{h\in U/\Z}a_h^p\prod_{i=0}^{p-1}\sigma^i(h)
\end{equation*}
is a sum of (mixed) terms of the form
\begin{equation*}
    \sum_{i=0}^{p-1} \sigma^i\left(\prod_{j=0}^{p-1} (a_{h_j}\sigma^j(h_j)) \right)
\end{equation*}
with $h_j\in U/\Z$, such that $\{h_j : 0\leq j\leq p-1\}$ consists of at least two elements, and hence belongs to $\im(\sigma^V_U)$, we can continue our above congruences by
\begin{align*}
\phantom{\ver^V_U(\pr^{V}_{V^{\ab}}\bar{x} )} & \equiv \pr^U_{U^{\ab}}\prod_{i=0}^{p-1}\left(\sum_{h\in
U/\Z} a_h\sigma^i(h) \right) \\
& = \prod_{i=0}^{p-1}\sigma^i(\pr^U_{U^{\ab}}(x_k))
\end{align*}
as desired.

\noindent
 4) The last condition requires further preparations. Let $U \in C(\G,\Z)$. If $U = \Z$ we put
\begin{align*}
    \alpha_\Z : \Lambda(\Z)^\times & \longrightarrow 1 + p\Lambda(\Z) \subseteq \Lambda(\Z)^\times \\
    f & \longmapsto \frac{f^p}{\varphi_\Z(f)} \ .
\end{align*}
If $U \neq \Z$ then we let $U' \in C(\G,\Z)$ be the unique subgroup of $U$ such that $[U:U']=p$. Since $U$ is abelian we may consider the composed homomorphism
\begin{equation*}
    \Lambda(U)^\times \xrightarrow{\; N^U_{U'} \;} \Lambda(U')^\times \xrightarrow{\; \subseteq \;} \Lambda(U)^\times \ .
\end{equation*}
It is shown in \cite{SV2} Prop.\ 2.3 (under more general circumstances) that modulo $p$ this map coincides with the map $f \longmapsto f^p$. It follows that
\begin{align*}
    \alpha_U : \Lambda(U)^\times & \longrightarrow 1 + p\Lambda(U) \subseteq \Lambda(U)^\times \\
    f & \longmapsto \frac{f^p}{N^U_{U'}(f)}
\end{align*}
is a well defined homomorphism.

We also need, for any open subgroup $U \subseteq \G$, the ring
\begin{equation*}
    \Lambda^\infty(U) := \varprojlim K[U/N]
\end{equation*}
where $N$ runs over all open normal subgroups of $U$ and where $K$ denotes the field of fractions of $\mathcal{O}$. We obviously have $\Lambda(U) \subseteq \Lambda^\infty(U)$. The usual logarithm series induces a homomorphism
\begin{multline*}
    \log : K_1(\Lambda(U)) = \Lambda(U)^\times / [\Lambda(U)^\times, \Lambda(U)^\times] \longrightarrow \\
     \Lambda^\infty(U)^{\ab} := \Lambda^\infty(U) / \overline{[\Lambda^\infty(U), \Lambda^\infty(U)]} \ .
\end{multline*}
Our maps $\tr^V_U$, $\eta_U$, and $\varphi_U$ extend in an obvious way to maps, denoted by the same symbols, between the $\Lambda^\infty(U)^{\ab}$. As a consequence of the first step in the proof of \cite{OT} Thm.\ 1.4 we have, for any open subgroups $U \subseteq V \subseteq \G$, the commutative diagram
\begin{equation}\label{d:trlogN}
    \xymatrix{
      K_1(\Lambda(V)) \ar[d]_{N^V_U} \ar[r]^{\log} & \Lambda^\infty(V)^{\ab} \ar[d]^{\tr^V_U} \\
      K_1(\Lambda(U)) \ar[r]^{\log} & \Lambda^\infty(U)^{\ab} .  }
\end{equation}

\begin{lemma}\label{log}
For any $U \neq \Z$ in $C(\G,\Z)$ the diagram
\begin{equation*}
    \xymatrix{
       \Lambda(U)^\times \ar[d]_{\alpha_U} \ar[r]^{\log} & \Lambda^\infty(U) \ar[d]^{p \eta_U} \\
       \Lambda(U)^\times \ar[r]^{\log} & \Lambda^\infty(U)   }
\end{equation*}
is commutative.
\end{lemma}
\begin{proof}
Let $U' \in C(\G,\Z)$ be the unique subgroup of $U$ of index $p$. Then
\begin{align*}
    \log (\alpha_U(f)) & = p \log(f) - \log(N^U_{U'}(f)) = p \log(f) - \tr^U_{U'}(\log(f)) \\
    & = (p \id - \tr^U_{U'})(\log(f)) \ .
\end{align*}
Hence we have to establish the identity
\begin{equation*}
    \id - \frac{1}{p} \tr^U_{U'} = \eta_U \ .
\end{equation*}
But, by definition, we have
\begin{equation*}
    \eta_U(g) =
    \begin{cases}
    g & \textrm{if $g \in U \setminus U'$}, \\
    0 & \textrm{if $g \in U'$}.
    \end{cases}
\end{equation*}
Remark \ref{trace}.ii says that the same computation holds for the left hand side.
\end{proof}

For $V \in S(\G,\Z)$ let $P(V)$ denote the set of all
$W\neq \Z$ in $C(\G,\Z)$ such that the unique subgroup $\Z\subseteq W'\subseteq W$ with $[W:W']=p$ is contained in $V$. We now introduce the map
\begin{equation*}
    \mathcal{L} = (\mathcal{L}_V)_V : [\prod_{U \in S(\G,\Z)} \Lambda(U^{\ab})^\times]_{(M3a)}  \longrightarrow \prod_{V \in S(\G,\Z)} \Lambda(V^{\ab}) \otimes_{\mathbb{Z}_p} \mathbb{Q}_p
\end{equation*}
defined by
\begin{equation*}
    \mathcal{L}_V((y_U)_U) := \frac{1}{p^2|V/\Z|}\log \big(
\frac{y_V^{p^2|V/\Z|}}{\varphi_\Z(y_\Z^p)\prod_{W\in P(V)} \varphi_W(\alpha_W(y_W))^{|W/\Z|}} \big)
\end{equation*}
where $[ \ldots ]_{(M3a)}$ indicates the subgroup of all those elements which satisfy \eqref{f:M3a}. The individual factors $\varphi_W(\alpha_W(y_W))$ appearing in the above definition lie in $1 + p\Lambda(W')$, since $\alpha_W(y_W) \in 1+p\Lambda(W)$, and hence can be viewed in $1+p\Lambda(V^{\ab})$. Moreover using \eqref{f:M3a} we have
\begin{align*}
\varphi_\Z(y_\Z^p)\prod_{W\in P(V)} \varphi_W(\alpha_W(y_W))^{|W/\Z|} & \equiv \varphi_\Z(y_\Z^p) \\
   & \equiv  y_\Z^{p^2} \\
   &\equiv  y_V^{p^2|V/\Z|} \mod p\Lambda(V^{\ab}) \ .
\end{align*}
Thus the logarithms in the asserted map are defined and lie in $p\Lambda(V^{\ab})$.

\begin{lemma}\label{M-A}
If $(y_U)_U \in [\prod_{U \in S(\G,\Z)}  \Lambda(U^{\ab})^\times]_{(M3a)}$ satisfies the condition \eqref{f:M1}, resp.\ \eqref{f:M2}, then $\mathcal{L}((y_U)_U)$ satisfies \eqref{f:A1}, resp.\ \eqref{f:A2}.
\end{lemma}
\begin{proof}
It is straightforward to check that \eqref{f:M2} implies \eqref{f:A2}. Let therefore $U\subseteq V$ be any two subgroups in $S(\G,\Z)$ such that $[V,V]\subseteq U$. We note that
\begin{multline*}
 \log(\frac{y_V^{p^2|V/\Z|}}{\varphi_\Z(y_\Z^p)\prod_{W\in P(V)}
\varphi_W(\alpha_W(y_W))^{|W/\Z|}})\\ =\log(\frac{y_V^{p^2|V/\Z|}}{\varphi_\Z(y_\Z^p)})-\log(\prod_{W\in P(V)}
\varphi_W(\alpha_W(y_W))^{|W/\Z|}) \ .
\end{multline*}
This will allow us to split the subsequent calculations into two parts. We begin by computing
\begin{align*}
 \tau^V_U\left(\tfrac{1}{p^2|V/\Z|}\log(
\frac{y_V^{p^2|V/\Z|}}{\varphi_\Z(y_\Z^p) })\right)
 & = \tfrac{1}{p^2|V/\Z|}\log(
\frac{\nu_U^V(y_V)^{p^2|V/\Z|}}{\nu_U^V(\varphi_\Z(y_\Z^p))}) \\
& = \tfrac{1}{p^2|V/\Z|}\log(
\frac{\pi_U^V(y_U)^{p^2|V/\Z|}}{ \varphi_\Z(y_\Z^p)^{|V/U|}}) \\
& = \pi_U^V\left(\tfrac{1}{p^2|U/\Z|}\log(
\frac{y_U^{p^2|U/\Z|}}{ \varphi_\Z(y_\Z^p) })\right)
\end{align*}
where the first, resp.\ second, identity uses \eqref{d:trlogN}, resp.\ \eqref{f:M1}. Next we compute
\begin{align*}
&  \tau^V_U\left(\tfrac{1}{p^2|V/\Z|}\log(\prod_{W\in P(V)}
\varphi_W(\alpha_W(y_W))^{|W/\Z|})\right)\\
& \qquad\qquad =\tfrac{1}{p^2|V/\Z|}\sum_{W\in P(V)}\tau^V_U( \log(\varphi_W(\alpha_W(y_W))^{|W/\Z|}))\\
& \qquad\qquad =\tfrac{1}{p |V/\Z|}\sum_{W\in P(V)} \tau^V_U \varphi_W\eta_W\log(y_W^{|W/\Z|})
\end{align*}
where the last identity uses Lemma \ref{log}. Using the definitions and Remark \ref{trace}.iii it is straightforward to check that
\begin{equation*}
    \tau^V_U \varphi_W\eta_{W}=
    \begin{cases} [V:U]\pi^V_U\varphi_W\eta_{W} & \hbox{if $W'\subseteq U$,} \\
     0, & \hbox{otherwise}
    \end{cases}
\end{equation*}
where, as before, $W'$ denotes the unique subgroup of $W$ of index $p$ and containing $\Z$. We therefore may continue the above computation by
\begin{align*}
    & \qquad\qquad =\tfrac{1}{p|V/\Z|}\sum_{W\in P(U)}  [V:U]\pi^V_U \varphi_W\eta_{W}\log(y_W^{|W/\Z|})\\
& \qquad\qquad =\pi^V_U\left(\tfrac{1}{p^2|U/\Z|}\sum_{W\in P(U)}  \log(\varphi_W(\alpha_W(y_W))^{|W/\Z|})\right)\\
& \qquad\qquad  = \pi^V_U\left(\tfrac{1}{p^2|U/\Z|}\log(\prod_{W\in P(U)}
\varphi_W(\alpha_W(y_W))^{|W/\Z|})\right)
\end{align*}
where the second identity again uses Lemma \ref{log}.
This establishes that
\begin{equation*}
    \tau^V_U(\mathcal{L}_V((y_U)_U))=\pi_U^V \mathcal{L}_U((y_U)_U) \ ,
\end{equation*}
i.\ e., the condition \eqref{f:A1} for $\mathcal{L}_V((y_U)_U)$.
\end{proof}

\begin{lemma}\label{beta}
For any $U \in S(\G,\Z)$ and any $f \in \Lambda^\infty(\G)^{\ab}$ we have:
\begin{equation*}
    \beta_U(\varphi_\G(f)) = \tfrac{1}{[U:\Z]} \varphi_\Z(\beta_\Z(f)) + \sum_{W \in P(U)} \tfrac{[W:\Z]}{[U:\Z]} \varphi_W(\eta_W(\beta_W(f))) \ .
\end{equation*}
\end{lemma}
\begin{proof}
It suffices to consider elements of the form $f = [g]_\G$ for some $g \in \G$. First let $g \not\in \Z$. We compute
\begin{align*}
    \beta_U(\varphi_\G([g]_\G)) & = \pr^U_{U^{\ab}}(\tr^\G_U ([g^p]_\G))
    = \pr^U_{U^{\ab}}(\sum_{h \in \G/U, h^{-1}g^ph \in U} [h^{-1}g^ph]_U) \\
    & = \sum_{h \in \G/U, h^{-1}g^ph \in U} h^{-1}g^ph [U,U] \\
    & = \tfrac{1}{[U:\Z]} \sum_{h \in \G/\Z, h^{-1}g^ph \in U} h^{-1}g^ph [U,U] \\
    & = \tfrac{1}{[U:\Z]} \sum_{W \in P(U)} \  \sum_{h \in \G/\Z, W=<h^{-1}gh,\Z>} h^{-1}g^ph [U,U] \\
    & = \tfrac{1}{[U:\Z]} \sum_{W \in P(U)} \varphi_W \big( \sum_{h \in \G/\Z, W=<h^{-1}gh,\Z>} h^{-1}gh \big) \\
    & = \tfrac{1}{[U:\Z]} \sum_{W \in P(U)} \varphi_W \big([W:\Z] \sum_{h \in \G/W, W=<h^{-1}gh,\Z>} h^{-1}gh \big) \\
    & = \tfrac{[W:\Z]}{[U:\Z]} \sum_{W \in P(U)} \varphi_W(\eta_W(\beta_W([g]_\G))) \ .
\end{align*}
We also have $\varphi_\Z(\beta_\Z([g]_\G)) = 0$ in this case. Now suppose that $g \in \Z$. Then the last term in the above computation vanishes. But we have
\begin{equation*}
    \beta_U(\varphi_\G([g]_\G)) = [\G:U] g^p [U,U] = \frac{1}{[U:\Z]} \varphi_\Z(\beta_\Z([g]_\G)) \ .
\end{equation*}
\end{proof}

Oliver and Taylor (cf.\ \cite{Oli} Chap.\ 6 or \cite{CR} \S54) have shown that
\begin{align*}
    L=L_\G : K_1(\Lambda(\G)) & \longrightarrow \mathcal{O}[[\Conj(\G)]] \\
    x & \longmapsto \log(x) - \frac{1}{p} \varphi_\G (\log(x))
\end{align*}
is a well defined homomorphism which makes the diagram
\begin{equation*}
    \xymatrix{
  K_1(\Lambda(\G)) \ar[r]^{\log} \ar[dr]_{L}
                & \Lambda^\infty(\G)^{\ab} \ar[r]^{\id - \frac{1}{p} \varphi_\G} &    \Lambda^\infty(\G)^{\ab}    \\
                & \mathcal{O}[[\Conj(\G)]]  \ar[ur]_{\rm can}               }
\end{equation*}
commutative. We observe that, since $\Lambda^\infty(\G)^{\ab}$ also can be viewed as the projective limit $\varprojlim K[\Conj(\G/N)]$, the right oblique arrow is injective.

\begin{proposition}\label{beta-L}
The diagram
\begin{equation*}
    \xymatrix{
      K_1(\Lambda(\G)) \ar[d]_{\theta} \ar[r]^{L} & \mathcal{O}[[\Conj(\G)]] \ar[d]^{\beta} \\
      [\prod_{U \in S(\G,\Z)} \Lambda(U^{\ab})^\times]_{(M3a)} \ar[r]^-{\mathcal{L}} & \prod_{U \in S(\G,\Z)} \Lambda(U^{\ab}) \otimes_{\mathbb{Z}_p} \mathbb{Q}_p  }
\end{equation*}
is commutative.
\end{proposition}
\begin{proof}
Let $x \in K_1(\Lambda(\G))$ and put $(x_U)_U := \theta(x)$. Using \eqref{d:trlogN} we obtain
\begin{equation*}
    \beta_V(\log(x)) = \log(\theta_V(x)) = \log(x_V)
\end{equation*}
for any $V \in S(\G,\Z)$. Using this identity together with Lemmas \ref{log} and \ref{beta} we compute
\begin{align*}
    & \beta_U(L(x)) = \beta_U(\log(x)) - \tfrac{1}{p}\beta_U(\varphi_\G(\log(x))) \\
    & = \log(x_U) - \tfrac{1}{p[U:\Z]}\varphi_\Z(\beta_\Z(\log(x))) - \sum_{W \in P(U)} \tfrac{[W:\Z]}{p[U:\Z]} \varphi_W(\eta_W(\beta_W(\log(x)))) \\
    & = \log(x_U) - \tfrac{1}{p[U:\Z]}\varphi_\Z(\log(x_\Z)) - \sum_{W \in P(U)} \tfrac{[W:\Z]}{p[U:\Z]} \varphi_W(\eta_W(\log(x_W))) \\
    & = \log(x_U) - \tfrac{1}{p[U:\Z]}\varphi_\Z(\log(x_\Z)) - \sum_{W \in P(U)} \tfrac{[W:\Z]}{p^2[U:\Z]} \varphi_W(\log(\alpha_W(x_W))) \\
    & = \tfrac{1}{p^2[U:\Z]} \big( \log(x^{p^2[U:\Z]}_U) - \log(\varphi_\Z(x^p_\Z)) - \sum_{W \in P(U)} \log(\varphi_W(\alpha_W(x_W))^{[W:\Z]}) \big) \\
    & = \mathcal{L}((x_U)_U) = \mathcal{L}(\theta(x))\ .
\end{align*}
\end{proof}

For any $V \in C(\G,\Z)$ we let $P_c(V)$ be the set of all $W \in C(\G,\Z)$ such that $V \subseteq W$ and $[W:V] = p$. We claim that
\begin{equation}\tag{M4}\label{f:M4}
    \alpha_V(x_V) - \prod_{W \in P_c(V)} \varphi_W(\alpha_W(x_W)) \in p \im(\sigma_V) \quad\textrm{for any $V \in C(\G,\Z)$}.
\end{equation}
holds true. We repeat that the image of the homomorphism $\varphi_W$, by construction, is contained in the subring $\Lambda(V)$. Hence the left hand side of \eqref{f:M4} indeed lies in $\Lambda(V)$.

\begin{lemma}\label{add-mult}
Let $(y_U)_U \in \prod_{U \in C(\G,\Z)} \Lambda(U)^\times$ and $V \in C(\G,\Z)$; then:
\begin{itemize}
  \item[i.] We have
\begin{multline*}
    \eta_V \big( \tfrac{1}{p^2|V/\Z|}\log \big(
\frac{y_V^{p^2|V/\Z|}}{\varphi_\Z(y_\Z^p)\prod_{W\in P(V)} \varphi_W(\alpha_W(y_W))^{|W/\Z|}} \big) \big) = \\ \tfrac{1}{p} \log \big( \frac{\alpha_V(y_V)}{\prod_{W \in P_c(V)} \varphi_W(\alpha_W(y_W))} \big) ;
\end{multline*}
  \item[ii.] if $(y_U)_U$ satisfies \eqref{f:M2} then the following assertions are equivalent:
\begin{itemize}
  \item[a.] $\alpha_V(y_V) - \prod_{W \in P_c(V)} \varphi_W(\alpha_W(y_W)) \in p \im(\sigma_V)$;
  \item[b.]
  $\frac{\alpha_V(y_V)}{\prod_{W \in P_c(V)} \varphi_W(\alpha_W(y_W))} \in 1 + p \im(\sigma_V)$;
  \item[c.] $\log ( \frac{\alpha_V(y_V)}{\prod_{W \in P_c(V)} \varphi_W(\alpha_W(y_W))} ) \in p \im(\sigma_V)$.
\end{itemize}
\end{itemize}
\end{lemma}
\begin{proof} i. We check that
\begin{equation*}
    \eta_V \left(\tfrac{1}{|V/\Z|}\log(
\frac{y_V^{p|V/\Z|}}{ \varphi_\Z(y_\Z) })\right)  = \log(\alpha_V(y_V))
\end{equation*}
and
\begin{multline*}
    \eta_V \left(\tfrac{1}{p|V/\Z|}\log(\prod_{W\in P(V)}
\varphi_W(\alpha_W(y_W))^{|W/\Z|})\right)  = \\  \log( \prod_{W \in P_c(V)} \varphi_W(\alpha_W(y_W)))
\end{multline*}
hold true. In case $V = \Z$ we have $\eta_\Z = \id_{\Lambda(\Z)}$, $\alpha_\Z(y_\Z) = \frac{y_\Z^p}{\varphi_\Z(y_\Z)}$ and $P(\Z)=P_c(\Z)$ which makes the two identities obvious. If $V \neq \Z$ we may use Lemma \ref{log} and further reduce to the identities
\begin{equation*}
    \tfrac{1}{p|V/\Z|}\log(
\frac{\alpha_V(y_V)^{p|V/\Z|}}{ \alpha_V(\varphi_\Z(y_\Z)) }) = \log(\alpha_V(y_V))
\end{equation*}
and
\begin{equation*}
    \sum_{W\in P(V)} \tfrac{|W/\Z|}{p|V/\Z|}
\eta_V(\varphi_W(\eta_W(\log(y_W)))) = \\  \sum_{W \in P_c(V)} \varphi_W(\eta_W(\log(y_W))) \ .
\end{equation*}
The first one follows from $\alpha_V \circ \varphi_\Z = 1$ and the second one from
\begin{equation*}
    \eta_V \circ \varphi_W \circ \eta_W =
    \begin{cases}
    \varphi_W \circ \eta_W & \textrm{if $W \in P_c(V)$}, \\
    0 & \textrm{if $W \in P(V) \setminus P_c(V)$.}
    \end{cases}
\end{equation*}
These latter equations are straightforward from the definitions.

ii. We begin by observing that $\alpha_V(y_V)$ and $\prod_{W \in P_c(V)} \varphi_W(\alpha_W(y_W))$ both are units in $1+p\Lambda(V)$. Hence the fraction in b. lies in $1+p\Lambda(V)$ so that its logarithm in c. is defined and contained in $p\Lambda(V)$ (recall that we assume $p \neq 2$). Since $p\im(\sigma_V)$ is an ideal in $\Lambda(V)^{W(V)}$ it suffices, for the equivalence of a. and b., to see that $\prod_{W \in P_c(V)} \varphi_W(\alpha_W(y_W)) \in \Lambda(V)^{W(V)}$. But this follows from \eqref{f:M2} since any $g \in W(V)$ permutes the elements of the set $P_c(V)$. For the equivalence of b. and c. it suffices to show that the isomorphisms
\begin{equation*}
    \xymatrix{
      1+p\Lambda(V) \ar[r]^-{\log} & p\Lambda(V) \ar@<1ex>[l]^-{\exp}  }
\end{equation*}
restrict to isomorphisms
\begin{equation*}
    \xymatrix{
      1+p\im(\sigma_V) \ar[r]^-{\log} & p\im(\sigma_V) \ar@<1ex>[l]^-{\exp} \ .  }
\end{equation*}
Consider
\begin{equation*}
    \log(1+p\sigma_V(f)) = \sum_{i \geq 1} (-1)^{i-1} \frac{p^i \sigma_V(f)^i}{i}
\end{equation*}
and
\begin{equation*}
    \exp(p\sigma_V(f)) = 1 + \sum_{i \geq 1} \frac{p^i \sigma_V(f)^i}{i!} \ .
\end{equation*}
Since $\im(\sigma_V)$ is an ideal in $\Lambda(V)^{W(V)}$ we have $\sigma_V(f)^i \in \im(\sigma_V)$ for any $i \geq 1$. Moreover, $\frac{p^i}{i!} \in p\mathbb{Z}_p$. Hence each summand in the above two series lies in $p\im(\sigma_V)$. As the latter is closed in $\Lambda(V)$ we conclude that
\begin{equation*}
    \log(1+p\sigma_V(f)) \in p\im(\sigma_V) \quad\textrm{and}\quad  \exp(p\sigma_V(f)) \in 1+p\im(\sigma_V) \ .
\end{equation*}
\end{proof}

Prop.\ \ref{beta-L} implies that $(\mathcal{L}_V((x_U)_U))_V$ satisfies $\eqref{f:A1}$ and that
\begin{equation*}
    \mathcal{L}_V((x_U)_U) \in \im(\sigma_V) \qquad\textrm{for any $V \in C(\G,\Z)$}.
\end{equation*}
Using Remark \ref{eta-A3}.ii.b and Lemma \ref{add-mult}.i we deduce that
\begin{equation*}
    p \eta_V( \mathcal{L}_V((x_U)_U) ) = \log ( \frac{\alpha_V( x_V)}{\prod_{W \in P_c(V)} \varphi_W(\alpha_W( x_W))} ) \in p \im(\sigma_V)
\end{equation*}
for any $V \in C(\G,\Z)$. Hence Lemma \ref{add-mult}.ii implies that \eqref{f:M4} holds true for $(x_U)_U = \theta(x)$.

Let $\Phi := \Phi^\G_\Z \subseteq \prod_{\Z \subseteq U \subseteq \G} \Lambda(U^{\ab})^\times$ be the subgroup of all elements $(x_U)_U$ which satisfy, for all $U \subseteq V$ in $S(\G,\Z)$, the conditions
\begin{align}
    & \nu^V_U (x_V) = \pi^V_U (x_U) \quad\textrm{if $[V,V] \subseteq U$}, \tag{\ref{f:M1}} \\
    & x_{gUg^{-1}} = gx_U g^{-1} \quad\textrm{for any $g \in \G$}, \tag{\ref{f:M2}} \\
    & \ver^V_U(x_V)- x_U \in  \im(\sigma_U^V) \quad\textrm{if $[V:U]=p$, and} \tag{\ref{f:M3}} \\
    & \alpha_U(x_U) - \prod_{W \in P_c(U)} \varphi_W(\alpha_W(x_W)) \in p \im(\sigma_U) \quad\textrm{if $U \in C(\G,\Z)$}. \tag{\ref{f:M4}}
\end{align}
So far we have established that
\begin{equation*}
    \im(\theta) \subseteq \Phi \ .
\end{equation*}

\begin{lemma}\label{calL-Phi}
$\mathcal{L}(\Phi) \subseteq \Psi$.
\end{lemma}
\begin{proof}
Let $(y_U)_U \in \Phi$ and put $(a_V)_V := \mathcal{L}((y_U)_U)$. We already know from Lemma \ref{M-A} that $(a_V)_V \in \Psi_{\mathbb{Q}_p}$. Lemma 4.5 and \eqref{f:M4} imply that
\begin{equation*}
    \eta_V(a_V) \in \im(\sigma_V) \qquad\textrm{for any $V \in C(\G,\Z)$}.
\end{equation*}
At this point we have to go back to Remark \ref{eta-A3}.ii.a and observe that the identity there equally holds for elements in $\prod_{V \in C(\G,\Z)} \Lambda(V) \otimes_{\mathbb{Z}_p} \mathbb{Q}_p$ satisfying \eqref{f:A1}. We check by induction with respect to the order of $V/\Z$ that
\begin{equation*}
    a_V \in \im(\sigma_V) \qquad\textrm{for any $V \in C(\G,\Z)$}.
\end{equation*}
For $V = \Z$ we have $\eta_\Z = \id$ and the claim is trivial. Let $V \neq \Z$ and let $\Z \subseteq W \subseteq V$ be the unique subgroup of index $p$. by the induction hypothesis we have $a_W \in \im(\sigma_W)$. Remark \ref{eta-A3}.i then implies that $\frac{1}{p} a_W \in \im(\sigma_V)$. Hence the identity $a_V = \eta_V(a_V) + \frac{1}{p} a_W$ shows that $a_V \in \im(\sigma_V)$. We now may apply Remark \ref{qp-version} to see that $(a_V)_V \in \Psi$.
\end{proof}

It follows that we have the commutative diagram
\begin{equation*}
    \xymatrix{
      K_1(\Lambda(\G)) \ar[d]_{\theta} \ar[r]^-{L} & \mathcal{O}[[\Conj(\G)]] \ar[d]^{\beta}_{\cong} \\
      \Phi \ar[r]^{\mathcal{L}} & \Psi   }
\end{equation*}
where the right perpendicular arrow, assuming also \eqref{f:H3}, is an isomorphism by Thm.\ \ref{iso-additive}. We define
\begin{equation*}
    SK_1(\Lambda(\G)) := \ker \big( K_1(\Lambda(\G)) \longrightarrow K_1(\Lambda^\infty(\G)) \big)
\end{equation*}
and
\begin{equation*}
    K'_1(\Lambda(\G)) := K_1(\Lambda(\G)) / SK_1(\Lambda(\G)) \ .
\end{equation*}
For any open subgroup $U \subseteq \G$ we consider the diagram
\begin{equation*}
    \xymatrix{
      K_1(\Lambda(\G)) \ar[d]_{N^\G_U} \ar[r] & K_1(\Lambda^\infty(\G)) \ar[d]^{N^\G_U} \\
      K_1(\Lambda(U)) \ar[d] \ar[r] & K_1(\Lambda^\infty(U)) \ar[d] \\
      K_1(\Lambda(U^{\ab}))  \ar[r] & K_1(\Lambda^\infty(U^{\ab}))  \\
      \Lambda(U^{\ab})^\times \ar[u]^{\cong} \ar[r] & \Lambda^\infty(U^{\ab})^\times \ar[u]_{\cong} .  }
\end{equation*}
The uppermost square commutes since
\begin{equation*}
    \Lambda(\G) \otimes_{\Lambda(U)} \Lambda^\infty(U) = \Lambda^\infty(\G) \ .
\end{equation*}
The two lower squares commute for trivial reasons. The indicated isomorphism on the right hand side is a special case of \cite{SV2} Prop.\ 3.1. Since the lowermost horizontal arrow visibly is injective we conclude that
\begin{equation*}
    SK_1(\Lambda(\G)) \subseteq \ker(\theta_U) \qquad\textrm{for any $\Z \subseteq U \subseteq \G$}.
\end{equation*}
Hence $\theta$ factorizes through a homomorphism
\begin{equation*}
    \theta : K'_1(\Lambda(\G)) \longrightarrow \Phi \ .
\end{equation*}
According to \cite{SV2} Cor.\ 3.2 we have
\begin{equation*}
    SK_1(\Lambda(\G)) = \varprojlim SK_1(\mathcal{O}[\G/N]) \ .
\end{equation*}
It therefore follows from \cite{Oli} Thm.\ 6.6 and Thm.\ 7.3 that $SK_1(\Lambda(\G))$ also lies in the kernel of the integral logarithm $L$ and that, more precisely, we have the exact sequence
\begin{equation}\label{f:exactsequ}
     1 \rightarrow \mu(\mathcal{O}) \times \G^{\ab} \xrightarrow{\; \iota \;} K'_1(\Lambda(\G)) \xrightarrow{\; L \;} \mathcal{O}[[\Conj(\G)]] \xrightarrow{\; \omega \;} \G^{\ab} \rightarrow 1
\end{equation}
where $\mu(\mathcal{O}) \subseteq \mathcal{O}^\times$ denotes the subgroup of all roots of unity, $\iota$ is the obvious map, and $\omega$ is given by
\begin{equation*}
    \omega(a[g]_\G) := g^{\Tr_{\mathcal{O}/\mathbb{Z}_p} (a)} [\G, \G] \qquad \textrm{for any $a \in \mathcal{O}$ and $g \in \G$}.
\end{equation*}
Let us now contemplate the commutative diagram
\begin{equation*}
    \xymatrix{
       1  \ar[r] & \mu(\mathcal{O}) \times \G^{\ab} \ar[d]_{=} \ar[r]^-{\iota} & K'_1(\Lambda(\G)) \ar[d]_{\theta} \ar[r]^-{L} & \mathcal{O}[[\Conj(\G)]] \ar[d]_{\beta}^{\cong} \ar[r]^-{\omega} & \G^{\ab} \ar[d]_{=} \ar[r] & 1  \\
     1 \ar[r] & \mu(\mathcal{O}) \times \G^{\ab} \ar[r]^-{\theta\circ \iota} & \Phi \ar[r]^-{\mathcal{L}} & \Psi \ar[r]^-{\omega\circ\beta^{-1}} & \G^{\ab} \ar[r] & 1 .  }
\end{equation*}
Our goal is to show that the lower row is exact as well, which then implies that $\theta$ is an isomorphism. Clearly $\theta_\G \circ\iota: \mu(\mathcal{O}) \times \G^{\ab} \longrightarrow \Lambda(\G^{\ab})^\times$ is injective. Hence $\theta \circ \iota$ is injective, and satisfies $\im(\theta\circ\iota) \subseteq \ker(\mathcal{L})$. For trivial reasons $\omega \circ \beta^{-1}$ is surjective with $\im(\mathcal{L}) \supseteq \ker(\omega\circ \beta^{-1})$. It therefore remains to establish the following two facts:
\begin{itemize}
  \item[a.] $\ker(\mathcal{L}|\Phi) \subseteq \im (\theta\circ \iota)$;
  \item[b.]$\im(\mathcal{L}|\Phi) \subseteq \ker(\omega\circ \beta^{-1})$.
\end{itemize}
But first of all we observe that, if $\G$ is abelian, then $\theta_\G$ and consequently by \eqref{f:M1a}, also $\theta : K_1(\Lambda(\G)) \xrightarrow{\; \cong \;} \Phi$ are isomorphisms. In particular, a. and b. hold in this case.

\begin{lemma}\label{LG}
For $(y_U)_U \in \Phi$ we have $\mathcal{L}_\G((y_U)_U) = \frac{1}{p} \log (\frac{y_\G^p}{\varphi_{\G^{\ab}}(y_\G)})$.
\end{lemma}
\begin{proof}
As a consequence of Lemmas \ref{inverse} and \ref{image-delta} any $(a_U)_U \in \Psi$ satisfies
\begin{equation*}
    a_\G = \sum_{V \in C(\G,\Z)} \tfrac{1}{[\G:V]} \eta_V(a_V) \qquad\textrm{in $\Lambda(\G^{\ab})$}.
\end{equation*}
Hence
\begin{equation*}
    \mathcal{L}_\G((y_U)_U) = \sum_{V \in C(\G,\Z)} \tfrac{1}{[\G:V]} \eta_V(\mathcal{L}_V((y_U)_U)) \qquad\textrm{in $\Lambda(\G^{\ab})$},
\end{equation*}
and, inserting the definition and using Lemma \ref{add-mult}.i, we obtain
\begin{align*}
    & \mathcal{L}_\G((y_U)_U) = \sum_{V \in C(\G,\Z)} \tfrac{1}{p|\G/V|} \log \big( \tfrac{\alpha_V(y_V)}{\prod\limits_{W \in P_c(V)} \varphi_W(\alpha_W(y_W))} \big) \\
    &  = \tfrac{1}{p|\G/\Z|} \log \big( \tfrac{\alpha_\Z(y_\Z)}{\prod\limits_{|W/\Z|=p} \varphi_W(\alpha_W(y_W))} \big) \\
    & \qquad\qquad\qquad\qquad + \sum_{V \in C(\G,\Z), V \neq \Z} \tfrac{1}{p|\G/V|} \log \big( \tfrac{\alpha_V(y_V)}{\prod\limits_{W \in P_c(V)} \varphi_W(\alpha_W(y_W))} \big) \\
    &  = \tfrac{1}{p^2|\G/\Z|}
    \log \big( \tfrac{\alpha_\Z(y_\Z)^p}{\prod\limits_{|W/\Z|=p} \varphi_W(\alpha_W(y_W))^p} \cdot \prod_{V \in C(\G,\Z), V \neq \Z}  \tfrac{\alpha_V(y_V)^{p|V/\Z|}}{\prod\limits_{W \in P_c(V)} \varphi_W(\alpha_W(y_W))^{p|V/\Z|}} \big) \\
    &  = \tfrac{1}{p^2|\G/\Z|}
    \log \big( \tfrac{\prod\limits_{W \in C(\G,\Z)} \alpha_W(y_W)^{p|W/\Z|} }{ \prod\limits_{W \in C(\G,\Z), W \neq \Z} \varphi_W(\alpha_W(y_W))^{|W/\Z|}} \big)
\end{align*}
emphasizing that the computation takes place in $\Lambda(\G^{\ab})$. Comparing this to the definition
\begin{equation*}
    \mathcal{L}_\G((y_U)_U) = \tfrac{1}{p^2|\G/\Z|}\log \big(
\tfrac{y_\G^{p^2|\G/\Z|}}{\varphi_\Z(y_\Z^p)\prod\limits_{W\in C(\G,\Z), W \neq \Z} \varphi_W(\alpha_W(y_W))^{|W/\Z|}} \big)
\end{equation*}
leads to the identity
\begin{align*}
    \log
\tfrac{y_\G^{p^2|\G/\Z|}}{\varphi_\Z(y_\Z^p)} & = \log \prod\limits_{W \in C(\G,\Z)} \alpha_W(y_W)^{p|W/\Z|} \\ & = \log \big( \tfrac{y_\Z^{p^2}}{\varphi_\Z(y_\Z^p)} \cdot \prod\limits_{W \in C(\G,\Z), W \neq \Z} \alpha_W(y_W)^{p|W/\Z|} \big)
\end{align*}
and hence to
\begin{equation*}
     \log y_\G^{p|\G/\Z|} = \log \big( y_\Z^p \cdot \prod\limits_{W \in C(\G,\Z), W \neq \Z} \alpha_W(y_W)^{|W/\Z|} \big)
\end{equation*}
where, we repeat, all factors in the argument of $\log$ are viewed in $\Lambda(\G^{\ab})$. We therefore may apply the ring homomorphism $\varphi_{\G^{\ab}}$, and we get
\begin{equation*}
    \log \varphi_{\G^{\ab}}(y_\G)^{p|\G/\Z|} = \log \big( \varphi_\Z(y_\Z^p) \cdot \prod\limits_{W \in C(\G,\Z), W \neq \Z} \varphi_W(\alpha_W(y_W))^{|W/\Z|} \big) \ .
\end{equation*}
By inserting this back into the definition we finally arrive at
\begin{equation*}
    \mathcal{L}_\G((y_U)_U) = \tfrac{1}{p^2|\G/\Z|}\log \big(
\tfrac{y_\G^{p^2|\G/\Z|}}{\varphi_{\G^{\ab}}(y_\G)^{p|\G/\Z|}} \big) =  \tfrac{1}{p} \log (\tfrac{y_\G^p}{\varphi_{\G^{\ab}}(y_\G)}) \ .
\end{equation*}
\end{proof}

\begin{lemma}\label{fact-b}
b. is satisfied.
\end{lemma}
\begin{proof}
Let $(y_U)_U \in \Phi$. Using Lemma \ref{LG} we compute
\begin{align*}
    (\omega \circ \beta^{-1})(\mathcal{L}((y_U)_U)) & = \omega(\mathcal{L}_\G((y_U)_U))
     = \frac{1}{p} \log (\frac{y_\G^p}{\varphi_{\G^{\ab}}(y_\G)}) \\
    & = \omega(L_{\G^{\ab}}(y_\G))
    = 1 \ ,
\end{align*}
where the last identity holds by the exact sequence \eqref{f:exactsequ} for the group $\G^{\ab}$.
\end{proof}

\begin{lemma}
a. is satisfied.
\end{lemma}

\begin{proof}
Let $(x_U)_U$ be in $\ker(\mathcal{L}|\Phi)$. This in particular means by Lemma \ref{add-mult}.i that
\begin{equation*}
    \log \big(\frac{\alpha_U(x_U)}{\prod_{V \in P_c(U)} \varphi_V(\alpha_V(x_V))} \big) = 0
\end{equation*}
for any $U \in C(\G,\Z)$. Since the logarithm is taken of an element in $1+p \Lambda(U)$ we deduce that
\begin{equation*}
    \alpha_U(x_U) = \prod_{V \in P_c(U)} \varphi_V(\alpha_V(x_V))
\end{equation*}
for any $U \in C(\G,\Z)$. If $U$ maximal in $C(\G,\Z)$ then $P_c(U)$ is empty and hence $\alpha_U(x_U)=1$. By downward induction we obtain
\begin{equation}\label{f:alpha}
    \alpha_U(x_U)=1 \qquad\textrm{for any $U \in C(\G,\Z)$.}
\end{equation}
In particular, for $U=\Z$ we get
\begin{equation*}
    \frac{x_\Z^p}{\varphi_\Z(x_\Z)}=\alpha_\Z(x_\Z)=1 \ .
\end{equation*}
But $\Z$ is abelian. In this case the integral logarithm is
\begin{align*}
    L_\Z : K_1(\Lambda(\Z)) & \longrightarrow \Lambda(\Z) \\
    x & \longmapsto \frac{1}{p} \log \big( \frac{x^p}{\varphi_\Z(x)} \big) \ .
\end{align*}
We see that $x_\Z \in \ker(L_\Z) = \mu(\cO) \times \Z \subseteq K_1(\Lambda(\Z)$ by \eqref{f:exactsequ}. Next suppose that $[U:\Z]=p$. Then, using \eqref{f:M1a}, we get
\begin{equation*}
    x_U^p = N^U_\Z(x_U) = x_\Z \in \mu(\cO) \times \Z \subseteq \mu(\cO) \times U \subseteq K_1(\Lambda(U)) \ .
\end{equation*}
From \eqref{f:exactsequ} for $L_U$ we see that $\Lambda(U)^\times/(\mu(\cO)\times U)$ is torsion free. We conclude that $x_U \in \mu(\mathcal{O})\times U$. Again by induction it then follows that
\begin{equation*}
    x_U \in \mu(\cO)\times U \subseteq K_1(\Lambda(U)) \qquad\textrm{for any $U \in C(\G,\Z)$.}
\end{equation*}
Similarly as above we have the relation
\begin{equation*}
    \frac{x_U^{p^2|U/\Z|}}{\prod_{W\in P(U)} \varphi_W(\alpha_W(x_W))^{|W/\Z|}}=1
\end{equation*}
for all $U$ in $S(\G,\Z)$. But by \eqref{f:alpha}
the denominator is equal to one, i.\ e., $x_U^{p^2|U/\Z|}=1$. Again by \eqref{f:exactsequ} for $L_{U^{\ab}}$ we see that the only torsion of $\Lambda(U^{\ab})^\times $ is contained in $\mu(\mathcal{O})\times
U^{\ab}$. Hence $x_U=1$. We now have established that
\begin{equation*}
    (x_U)_U\in \prod_{U\in S(\G,\Z)} \big( \mu(\mathcal{O})\times U^{\ab}\big)\cap \Phi \ .
\end{equation*}
At this point we need the condition \eqref{f:M3}. We claim that it implies
\begin{equation*}
    x_U=\mathrm{ ver}^\G_U(x_\G) \qquad\textrm{for any $U \in S(\G,\Z)$}.
\end{equation*}
Indeed, let first $U$ be of index $p$ in $\G$ and assume that for $g,h\in U^{\ab}$ and $\zeta,\xi\in \mu(\mathcal{O})$ the element $\zeta g-\xi h$ lies in $\im(\sigma^\G_U)$. Then,
under the augmentation map $\epsilon$ of $\Lambda(U^{\ab})$ we have that $\epsilon(\zeta g-\xi
h)=\zeta-\xi\in \epsilon(\im(\sigma^\G_U))=|\G/U|\mathcal{O}$ is divisible by $p$. Since $\zeta$ and $\xi$  are roots of unity in an unramified extension of $\mathbb{Z}_p$ this implies that $\zeta=\xi$. As $\zeta$ is a unit in $\cO$ it follows that $g-h$ lies in $\im(\sigma^\G_U)$ and therefore is, as any element in $\im(\sigma^\G_U)$, invariant under the conjugation action of $\G/U$. We choose a set of representatives $R \subseteq U^{\ab}$ of the cosets in $U/\Z[U,U]$ and  write
\begin{equation*}
    g =z_g u_g\; ,\ h = z_h u_h,\ \textrm{and}\ g-h = \sigma^\G_U (\sum_{u \in R} a_u u)
\end{equation*}
with $u_g, u_h \in R$, $z_g,z_h \in \Z[U,U]/[U,U]$, and $a_u \in \Lambda(\Z)$. On the one hand the invariance of $g-h$ implies that $u_g$ and $u_h$ are both invariant as $p\neq 2$. On the other hand any $\tau \in \G/U$ induces a permutation $\tau$ of $R$ such that $\tau u \tau^{-1} = b_{\tau,u}\tau(u)$ with $b_{\tau,u} \in \Z$ for any $u \in R$. We compute
\begin{align*}
    z_g u_g - z_h u_h = g - h = \sigma^\G_U (g-h) & = \sum_{u \in R} \sum_{\tau \in \G/U} a_u \tau u\tau^{-1} \\
    & = \sum_{u \in R} \sum_{\tau \in \G/U} a_u b_{\tau,u} \tau (u) \\
    & = \sum_{u \in R} (\sum_{\tau \in \G/U} a_{\tau^{-1}(u)} b_{\tau,\tau^{-1}(u)}) u \ .
\end{align*}
Since $\tau(u_g) = u_g$, $\tau(u_h) = u_h$, and $b_{\tau,u_g} = b_{\tau,u_h} = 1$ for any $\tau$ it follows that
\begin{equation*}
    z_g = [\G:U] a_{u_g} = p a_{u_g} \quad\textrm{and}\quad z_h = [\G:U] a_{u_h} = p a_{u_h}
\end{equation*}
provided $u_g \neq u_h$. Since this is not possible we must have $u := u_g = u_h$. But then $z_g - z_h = p a_u$ which implies $z_g = z_h$ and hence $g=h$. This proves our claim for $U$ of index $p$ in $\G$. The general case then follows by induction. This together with \eqref{f:N-ver} shows that $(x_U)_U \in \im(\theta\circ\iota)$.
\end{proof}

\begin{theorem}\label{iso-multiplicative}
Assuming \eqref{f:H1} - \eqref{f:H4} the map $\theta : K'_1(\Lambda(\G)) \xrightarrow{\; \cong \;} \Phi^\G_\Z$ is an isomorphism.
\end{theorem}

\section{The integral logarithm for $B(\G)$}\label{sec:log-B}

In this section we will construct an extension of the integral logarithm $L$ of Oliver and Taylor to $K_1(B(\G))$. We assume \eqref{f:H2}, \eqref{f:H4}, and that $H$ is finite, which implies \eqref{f:H1}.
The ring $B(\G)$ being local we have the surjection
$B(\G)^\times \twoheadrightarrow K_1(B(\G))$. The usual computation for the convergence of the logarithm series
\begin{equation*}
    \log(x)=\mathrm{Log}(1+x)=\sum_{n\geq 1} \frac{(-1)^{n+1}}{n}x^n
\end{equation*}
shows that it induces a homomorphism
\begin{equation*}
    \log : 1+ \Jac(B(\G)) \longrightarrow (B(\G)/[B(\G),B(\G)])\otimes_{\mathbb{Z}_p}\mathbb{Q}_p \ .
\end{equation*}
The additional point to note is that the ideal $\Jac(B(\G))/p B(\G)$ is nilpotent in $B(\G)/p B(\G)$; in particular, the denominators appearing in the image of this map are bounded.

On the other hand, as $\varphi_\Z(S(\Z))\subseteq S(\Z)$ we obtain a unique extension   $\varphi_\Z:A(\Z) \longrightarrow A(\Z)$ of
$\varphi_\Z$. It respects $\Jac (A(\Z))$ and therefore further extends to a homomorphism of
rings $\varphi_\Z:B(\Z) \longrightarrow B(\Z)$, which modulo $p$ induces   the endomorphism of $ B(\Z)/p
B(\Z)\cong \mathcal{O}/(p)[[\Z]]$  which sends $f\mapsto f^p$. Furthermore, for any subgroup $U \in S(\G,\Z)$ we may use the identification $B(U)/[B(U),B(U)] = B(\Z)\otimes_{\Lambda(\Z)}\mathcal{O}[[\Conj(U)]]$
from Lemma \ref{B-comm}.ii to extend the earlier map $\varphi_U$ to the map
\begin{align*}
    \varphi_U: B(U)/[B(U),B(U)] & \longrightarrow   B(U)/[B(U),B(U)] \\
    z\otimes f & \longmapsto \varphi_\Z(z)\otimes \varphi_U(f) \ .
\end{align*}
If $U$ is abelian then $\varphi_U$
is again a ring endomorphism of $B(U)$ which modulo $p$ coincides with the map $f\mapsto f^p$.

\begin{lemma}\label{units}
The group $B(\G)^\times/ [1+pB(\G)] \cdot B(\Z')^\times$, for any open subgroup $\Z' \subseteq \G_0$, is annihilated by $p^\ell$ for a sufficiently large $\ell \in \mathbb{N}$.
\end{lemma}
\begin{proof}
Let $I_B := \ker (B(\G) \longrightarrow B(\Gamma))$. We also fix a section $\sigma : \Gamma \rightarrow \G$ of the projection map $\G \rightarrow \Gamma$. It induces a section $B(\Gamma) \xrightarrow{\; \cong \;} B(\sigma(\Gamma)) \subseteq B(\G)$ of the ring homomorphism $B(\G) \longrightarrow B(\Gamma)$. Hence
\begin{equation}\label{f:units}
    B(\G)^\times = [1+I_B] \cdot B(\sigma(\Gamma))^\times = [1+\Jac(B(\G))] \cdot B(\sigma(\Gamma))^\times \ .
\end{equation}
If $\Jac(B(\G))^{p^\ell} \subseteq pB(\G)$ then $[1+\Jac(B(\G))]^{p^\ell} \subseteq 1+pB(\G)$. On the other hand, the subgroup $\sigma(\Gamma) \cap \Z'$ is open and hence contains $\sigma(\Gamma)^{p^\ell}$ for some sufficiently large
$\ell \in \mathbb{N}$. It follows that $\varphi_{\sigma(\Gamma)}^\ell(\sigma(\Gamma)) \subseteq \Z'$ and therefore that
\begin{equation*}
    \varphi_{\sigma(\Gamma)}^\ell(B(\sigma(\Gamma))^\times) \subseteq B(\Z')^\times \ .
\end{equation*}

\pagebreak

\noindent
Since $ \varphi_{\sigma(\Gamma)}^\ell (x) \equiv x^{p^\ell} \bmod 1+pB(\sigma(\Gamma))$ for any $x \in B(\sigma(\Gamma))^\times$ we conclude that
\begin{equation*}
    B(\sigma(\Gamma))^{\times p^\ell} \subseteq [1+pB(\sigma(\Gamma))] \cdot B(\Z')^\times \subseteq [1+pB(\G)] \cdot B(\Z')^\times \ .
\end{equation*}
\end{proof}

We now define the homomorphism
\begin{align*}
   [1+\Jac(B(\G))] B(\G_0)^\times & \longrightarrow (B(\G)/[B(\G),B(\G)]) \otimes_{\mathbb{Z}_p} \mathbb{Q}_p \\
    x & \longmapsto
    \begin{cases}
    \log(x)-\tfrac{1}{p} \varphi_\G(\log(x)) & \textrm{if $x \in 1+\Jac(B(\G))$}, \\
    \tfrac{1}{p}\log(\tfrac{x^p}{\varphi_{\G_0} (x)}) & \textrm{if $x \in B(\G_0)^\times$}.
    \end{cases}
\end{align*}
This is well defined since:
\begin{itemize}
  \item[--] $\tfrac{x^p}{\varphi_{\G_0} (x)} \in 1+pB(\G_0)$ for $x \in B(\G_0)^\times$.
  \item[--] $\tfrac{1}{p}\log(\tfrac{x^p}{\varphi_{\G_0} (x)}) = \log(x)-\tfrac{1}{p} \varphi_\G(\log(x))$ for $x \in [1+\Jac(B(\G))] \cap B(\G_0)^\times$.
\end{itemize}
As a consequence of Lemma \ref{units} and the unique divisibility of the target this map extends uniquely to $B(\G)^\times$ and induces a natural homomorphism
\begin{equation*}
    L_B := L_{B(\G)} : K_1(B(\G)) \longrightarrow (B(\G)/[B(\G),B(\G)]) \otimes_{\mathbb{Z}_p} \mathbb{Q}_p \ .
\end{equation*}

\begin{proposition}\label{integral-LB}
$L_B$ has image in $B(\G)/[B(\G),B(\G)]$.
\end{proposition}
\begin{proof} First of all we note that $B(\G)/[B(\G),B(\G)]$ by Lemma \ref{B-comm}.ii  is $p$-torsion free so that the statement makes sense.

For $x \in B(\sigma(\Gamma))^\times$ we have that $L_B(x)$ is the image in $B(\G)/[B(\G),B(\G)]$ of
\begin{equation*}
    \tfrac{1}{p}\log(\tfrac{x^p}{\varphi_{\sigma(\Gamma)} (x)}) \in \tfrac{1}{p} \log(1+pB(\sigma(\Gamma))) \subseteq B(\sigma(\Gamma)) \ .
\end{equation*}
Therefore \eqref{f:units} reduces us to proving that
\begin{equation*}
    \log(x)-\tfrac{1}{p} \varphi_\G(\log(x)) \in B(\G)/[B(\G),B(\G)] \qquad\textrm{for any $x \in 1+\Jac(B(\G))$}.
\end{equation*}
This is done by arguments completely analogous to the case of the integral logarithm of Oliver and Taylor (cf.\ \cite{Oli} Chap.\ 6 or \cite{CR} \S54).
\end{proof}

\begin{lemma}\label{L-L}
If $i : \Lambda(\G) \rightarrow B(\G)$ denotes the inclusion map then the diagram
\begin{equation*}
    \xymatrix{
      K_1(\Lambda(\G)) \ar[d]_{L_\G} \ar[r]^{K_1(i)} & K_1(B(\G)) \ar[d]^{L_{B(\G)}} \\
      \cO[[\Conj(\G)]] \ar[r]^-{i} & B(\G)/[B(\G),B(\G)]   }
\end{equation*}
commutes.
\end{lemma}
\begin{proof}
If $I := \ker(\Lambda(\G) \longrightarrow \Lambda(\Gamma))$ then we have, as in the proof of Lemma \ref{units}, that $\Lambda(\G)^\times/ [1+I] \cdot \Lambda(\G_0)^\times$ is annihilated by some $p^\ell$. But
\begin{equation*}
    L_{B(\G)} \big| [1+I] = i \circ L_\G \big| [1+I]
\end{equation*}
by the very definitions and
\begin{equation*}
    L_{B(\G_0)} \big| \Lambda(\G_0) = L_{\G_0}
\end{equation*}
since $\G_0$ is commutative.
\end{proof}

Next we define
\begin{gather*}
    SK_1(A(\G)) := \ \textrm{image of $SK_1(\Lambda(\G))$ in $K_1(A(\G))$}, \\
    SK_1(B(\G)) := \ \textrm{image of $SK_1(\Lambda(\G))$ in $K_1(B(\G))$}
\end{gather*}
and
\begin{gather*}
    K'_1(A(\G)) := K_1(A(\G))/SK_1(A(\G)), \\
    K'_1(B(\G)) := K_1(B(\G))/SK_1(B(\G)).
\end{gather*}
Lemma \ref{L-L} implies that $L_{B(\G)}$ induces a homomorphism
\begin{equation*}
    L_B = L_{B(\G)} : K'_1(B(\G)) \longrightarrow B(\G)/[B(\G),B(\G)] \ .
\end{equation*}

\section{The multiplicative theory, part 2}\label{sec:multiplicative-B}

Throughout this section we assume \eqref{f:H1} - \eqref{f:H4} and, in fact, that $H$ is finite. We choose $\Z$ small enough so that it is torsion free (i.\ e., $\Z \cong \mathbb{Z}_p$). We shall study the analogs
\begin{align*}
    \theta_A := (\theta^\G_\Z)_A : K_1(A(\G)) & \longrightarrow \prod_{U \in S(\G,\Z)} A(U^{\ab})^\times
\end{align*}
and
\begin{align*}
    \theta_B := (\theta^\G_\Z)_B : K_1(B(\G)) & \longrightarrow \prod_{U \in S(\G,\Z)} B(U^{\ab})^\times
\end{align*}
of $\theta$, where we abbreviate
\begin{equation*}
    A:=A(\G):=\Lambda(\G)_{S(\G)}
\end{equation*}
and as before
\begin{equation*}
    B:=B(\G):=\widehat{\Lambda(\G)_S} \ .
\end{equation*}
Since by Prop.\ \ref{finitefree} we have
\begin{equation*}
    A(\cdot)=A(\Z)\otimes_{\Lambda(\Z)}\Lambda(\cdot) \qquad\text{and}\qquad B(\cdot)=B(\Z)\otimes_{\Lambda(\Z)}\Lambda(\cdot)
\end{equation*}
the norm maps $N^V_U $ extend naturally and one immediately checks that the analogs
\begin{equation*}
    (\theta_A)_U : K_1(A(\G)) \xrightarrow{\; N^\G_U \;} K_1(A(U)) \longrightarrow K_1(A(U^{\ab})) = A(U^{\ab})^\times
\end{equation*}
and similarly $(\theta_B)_U$  of $\theta_U$ are defined and induce a commutative diagram
\begin{equation*}
    \xymatrix{
  K_1(\Lambda(\G) \ar[d]_{} \ar[r]^{\theta_U} & {\Lambda(U^{\ab})^\times} \ar@{^{(}->}[d] \\
  K_1(A(\G)) \ar[d]_{ } \ar[r]^{(\theta_A)_U} & A( U^{\ab})^\times\ar@{^{(}->}[d]  \\
  K_1(B(\G)) \ar[r]^{(\theta_B)_U} & B( U^{\ab})^\times
 }
\end{equation*}
which implies an analogous diagram for $\theta$, $\theta_A$, and $\theta_B$.

Let $U \subseteq V$ be subgroups in $S(\G,\Z)$. As
$\Z\subseteq U^{\ab}$ we have $A(U^{\ab}) = A(\Z) \otimes_{\Lambda(\Z)}\Lambda(U^{\ab})$ and similarly for $B(U^{\ab})$. It follows easily that the maps
$\pi^V_U$, $\nu^V_U$, and $\sigma^V_U$ extend naturally to the rings $A(\cdot)$ and $B(\cdot)$. In particular, $\alpha_\Z$ extends
to a map $\alpha_\Z: B(\Z)^\times \to 1+pB(\Z)\subseteq B(\Z)^\times$ sending again $f$ to
$\frac{f^p}{\varphi_\Z(f)}$. If $U\in C(\G,\Z)$ differs from $\Z$ then the reasoning in the proof of \cite{SV2} Prop.\ 2.3 still works and shows that  the map
\begin{align*}
    \alpha_U:B(U)^\times & \longrightarrow B(U)^\times \\
    f & \longmapsto \frac{f^p}{N^U_{U'}(f)} \ ,
\end{align*}
which extends the earlier such map, has image in $1+pB(U)$.

If $[V:U] = p^n$ then the Verlagerung satisfies
\begin{equation*}
    \ver^V_U(z \bar{g}) = z^{p^n} \ver^V_U(\bar{g}) = \varphi_\Z^n(z) \ver^V_U(\bar{g})
\end{equation*}
for any $z \in \Z$ and $\bar{g} \in U^{\ab}$. Hence we may extend the earlier ring homomorphism $\ver^V_U : \Lambda(V^{\ab}) \longrightarrow \Lambda(U^{\ab})$ to the ring homomorphism
\begin{align*}
    B(V^{\ab}) = B(\Z) \otimes_{\Lambda(\Z)}\Lambda(V^{\ab}) & \longrightarrow B(U^{\ab}) = B(\Z) \otimes_{\Lambda(\Z)}\Lambda(U^{\ab}) \\
    z \otimes f & \longmapsto \varphi_\Z^n(z) \otimes \ver^V_U(f) \ .
\end{align*}
It maps $A(V^{\ab})$ into $A(U^{\ab})$.

We now may define the subsets
\begin{equation*}
    \Phi_A=(\Phi^\G_\Z)_A\subseteq \prod_{U\in S(\G,\Z)} A(U^{\ab})^\times
\end{equation*}
and
\begin{equation*}
    \Phi_B=(\Phi^\G_\Z)_B\subseteq \prod_{U\in S(\G,\Z)} B(U^{\ab})^\times
\end{equation*}
by the corresponding conditions \eqref{f:M1} - \eqref{f:M4}, and we obtain the obvious inclusions
\begin{equation*}
    \Phi\subseteq\Phi_A\subseteq \Phi_B \ .
\end{equation*}
For the purpose of   proving the main conjecture the following result is essential.

\begin{theorem}\label{imagethetaA}
\begin{enumerate}
\item[i.] The image of $\theta_A$ is contained in $\Phi_A$.
\item[ii.] Within $\prod_{U\in S(\G,\Z)} A(U^{\ab})^\times$ we have
    \begin{equation*}
        \Phi_A\cap \prod_{U\in S(\G,\Z)} \Lambda(U^{\ab})^\times =\Phi=\im(\theta) \ .
    \end{equation*}
\end{enumerate}
\end{theorem}

Note that Thm.\ \ref{imagethetaA}.ii follows trivially from the  definition of $\Phi$ and $\Phi_A$ as well as from Thm.\ \ref{iso-multiplicative}. While one checks that the image of $\theta_A$ satisfies
\eqref{f:M1} - \eqref{f:M3} for the same reasons as  $\im(\theta)$, condition \eqref{f:M4} would
again require an (integral) logarithm, which due to convergence issues is not available for
$A(\G)$. This is the {\it only reason} to consider also $\theta_B$ and $\Phi_B$ in the following!
Thus we are going to prove the following variant

\begin{theorem}\label{imagethetaB}
\begin{enumerate}
\item[i.] The image of $\theta_B$ is contained in $\Phi_B$.
\item[ii.] Within $\prod_{U\in S(\G,\Z)} B(U^{\ab})^\times$ we have
    \begin{equation*}
        \Phi_B\cap \prod_{U\in S(\G,\Z)} A(U^{\ab})^\times =\Phi_A \ .
    \end{equation*}
\end{enumerate}
\end{theorem}

Again, Thm.\ \ref{imagethetaB}.ii holds for trivial reasons. Moreover, Thm.\ \ref{imagethetaB}.i
clearly implies Thm.\ \ref{imagethetaA}.i. It is straightforward to check (by the same arguments as
before) that the image of $\theta_B$ satisfies the conditions \eqref{f:M1} - \eqref{f:M3}. In
particular,  upon replacing \eqref{f:decomp} by
\begin{equation*}
    B(V)\cong\bigoplus_{i=0}^{p-1}B(U)g^i
\end{equation*}
and using also
\begin{equation*}
    B(U)\cong \bigoplus_{h\in U/\Z}B(\Z) h
\end{equation*}
the proof for \eqref{f:M3} is literally the same as before.

Before we discuss the condition \eqref{f:M4} we would like to point out the following conclusion.

\begin{corollary}
For any compact $p$-adic Lie group $G$ and $\mathcal{O}$ the ring of integers of a finite unramified
extension of $\mathbb{Q}_p$ the canonical maps
\begin{equation*}
    K_1'(\Lambda(G))\hookrightarrow
K_1'(A(G)) \qquad\text{and}\qquad K_1'(\Lambda(G))\hookrightarrow
K_1'(B(G))
\end{equation*}
are injective.
\end{corollary}
\begin{proof}
Assuming first that $G=\G$ is pro-$p$ of dimension one, the claim follows from the following
commutative diagram
\begin{equation*}
    \xymatrix{
  K_1'(\Lambda(\G))\ar@{^{(}->}[d]_{\theta} \ar[r]^{ } & K_1'(A(\G)) \ar[d]^{\theta_A} \\
  {\Phi} \ar@{^{(}->}[r]^{ } & {\Phi_A}.   }
\end{equation*}
In the general pro-$p$ case we may express $\Lambda(G)$ as an inverse limit of Iwasawa algebras of groups $\G$ of dimension one which by \cite{FK} Prop.\ 1.5.1 and \cite{SV2} Cor.\ 3.2 gives rise to a commutative diagram
\begin{equation*}
    \xymatrix{
    K_1'(\Lambda(G)) \ar[d] \ar[r]^{\cong} & {\varprojlim  K_1'(\Lambda(\G))} \ar@{^{(}->}[d] \\
    K_1'(A(G)) \ar[r] & {\varprojlim  K_1'(A(\G))}}
\end{equation*}
and to a corresponding one for $K'_1(B(G))$. The claim follows. For general $G$ one uses the same reduction steps as discussed in \cite{Su}.
\end{proof}

In order to show that the image of $\theta_B$ satisfies \eqref{f:M4} we need the commutativity  of the diagram
\begin{equation}
\label{d:comm-L-L-B} \xymatrix{
  K_1(B(\G)) \ar[d]_{\theta_B} \ar[r]^-{L_B} & B(\G)/[B(\G),B(\G)] \ar[d]^{\beta_B} \\
  [\prod_{U\in S(\G,\Z)} B(U^{\ab})^\times]_{(M3a)} \ar[r]^-{\mathcal{L}_B} & {\prod\limits_{U\in S(\G,\Z)} B(U^{\ab})\otimes_{\mathbb{Z}_p}\mathbb{Q}_p} ,  }
\end{equation}
where $[\ldots]_{(M3a)}$ indicates, as before, the subgroup of elements satisfying \eqref{f:M3a} and where $\mathcal{L}_B =(\mathcal{L}_{B,V})_{V\in S(\G,\Z)}$ is defined in the same manner as earlier the map $\mathcal{L}$:
\begin{equation*}
    \mathcal{L}_{B,V}((y_U)_U) := \frac{1}{p^2|V/\Z|}\log \big(
\frac{y_V^{p^2|V/\Z|}}{\varphi_\Z(y_\Z^p)\prod_{W\in P(V)} \varphi_W(\alpha_W(y_W))^{|W/\Z|}} \big)
\ .
\end{equation*}

By $\Lambda(\Z)$-linearity the map $\eta_U$ extends to a $B(\Z)$-linear map
\begin{equation*}
    \eta_U: B(U)\to
B(U) \ .
\end{equation*}

\begin{lemma}\label{logB}
\begin{enumerate}
\item[i.] For any open subgroups $U \subseteq V \subseteq \G$, we have the commutative diagram
\begin{equation*}
    \xymatrix{
       \frac{1 + pB(V)}{[1+pB(V),B(V)^\times]} \ar[d]_{N^V_U} \ar[r]^-{\log} & B(V)/[B(V),B(V)] \otimes_{\mathbb{Z}_p}\mathbb{Q}_p \ar[d]^{\tr^V_U} \\
       \frac{1 + pB(U)}{[1+pB(U),B(U)^\times]} \ar[r]^-{\log} & B(U)/[B(U),B(U)] \otimes_{\mathbb{Z}_p}\mathbb{Q}_p .  }
\end{equation*}
\item[ii.]
For any $U \neq \Z$ in $C(\G,\Z)$ the diagram
\begin{equation*}
    \xymatrix{
       1 + pB(U)  \ar[d]_{\alpha_U} \ar[r]^{\log} & B(U) \otimes_{\mathbb{Z}_p}\mathbb{Q}_p \ar[d]^{p \eta_U} \\
       1 + pB(U) \ar[r]^{\log} & B(U) \otimes_{\mathbb{Z}_p}\mathbb{Q}_p   }
\end{equation*}
is commutative.
\end{enumerate}
\end{lemma}

\begin{proof}
ii.\ follows formally from i.\ in the same way as Lemma \ref{log} follows from \eqref{d:trlogN}. For i.\ one has to adapt the first step in the proof of \cite{OT} Thm.\ 1.4.
\end{proof}

\begin{remark}\label{betaB}
By $B(\Z)$-(semi)linearity the identity in Lemma \ref{beta} remains valid for any $f \in B(\G)/[B(\G),B(\G)] \otimes_{\mathbb{Z}_p}\mathbb{Q}_p$ and any $U \in S(\G,\Z)$.
\end{remark}

To establish the commutativity of \eqref{d:comm-L-L-B} we have to show that
\begin{equation*}
    \mathcal{L}_B(\theta_B(x)) = \beta_B(L_B(x)) \qquad\text{for any $x \in K_1(B(\G))$}
\end{equation*}
holds true. Because of Lemma \ref{units} it suffices to treat the two special cases $x \in 1+pB(\G)$ and $x \in B(\Z)^\times$. In the first case the computation is formally the same as the computation in the proof of Prop.\ \ref{beta-L}, now using \eqref{d:comm-L-L-B}, Lemma \ref{logB}, and Remark \ref{betaB} of course. Let therefore $x \in B(\Z)^\times$. Then $N^\G_U(x) = x^{|\G/U|}$ for any $U \in S(\G,\Z)$ and $\alpha_W(x) = \frac{x^p}{N^W_{W'}(x)} = \frac{x^p}{x^p} = 1$ for any $W \in P(U)$. Hence we obtain $\theta_B(x) = (x^{|\G/U|})_U$ and
\begin{align*}
    \mathcal{L}_B(\theta_B(x)) & = ( \tfrac{1}{p^2|V/\Z|}\log \big(
\frac{x^{p^2|\G/\Z|}}{\varphi_\Z(x^{p|\G/\Z|})\prod_{W\in P(V)} \varphi_W(\alpha_W(x))^{|\G/\Z|}} \big) )_V \\
& = ( \tfrac{|\G/V|}{p} \log \big( \frac{x^p}{\varphi_\Z(x)} \big) )_V \\
& = ( |\G/V| L_B(x) )_V \\
& = \beta_B(L_B(x)) \ .
\end{align*}

To see that $\theta_B(x)$ satisfies \eqref{f:M4} we observe that due to Thm.\ \ref{iso-B-additive} and the commutativity of \eqref{d:comm-L-L-B} the element $\mathcal{L}_B(\theta_B(x))=\beta_B(L_B(x))$ satisfies \eqref{f:A1} and \eqref{f:A3}. We further note that the obvious analogs of Remark \ref{eta-A3}.ii.b and Lemma \ref{add-mult} remain true in the present setting. Hence we may argue exactly as we did earlier (before Lemma \ref{calL-Phi}) for $\theta(x)$.

This finishes the proof of Theorems \ref{imagethetaB} and \ref{imagethetaA}.


\end{document}